\documentclass[a4paper,DIV=classic,fontsize=10pt]{scrartcl}
\KOMAoptions{DIV=15}
\usepackage{graphicx}

\usepackage{comment}
\usepackage[utf8]{inputenc}
\usepackage{amsmath,amssymb,amsthm}
\usepackage{
	,enumerate
}
\usepackage[usenames,dvipsnames]{color}
\usepackage[left=2.7cm,top=2.5cm,right=2.8cm,bottom= 5cm]{geometry}

\newtheorem{theorem}{Theorem}[section]

\newtheorem{corollary}[theorem]{Corollary}
\newtheorem{lemma}[theorem]{Lemma}
\newtheorem{remark}[theorem]{Remark}
\newtheorem{example}[theorem]{Example}
\newtheorem{examples}[theorem]{Examples}
\newtheorem{foo}[theorem]{Remarks}

\newcommand{\eins}{\text{\ensuremath{1\hspace*{-0.9ex}1}}}
\newcommand{\Lip}{\text{Lip}_1}

\newcommand{\be}{\begin{equation}}
\newcommand{\ee}{\end{equation}}
\newcommand{\bea}{\begin{eqnarray}}
\newcommand{\eea}{\end{eqnarray}}
\newcommand{\beas}{\begin{eqnarray*}}
	\newcommand{\eeas}{\end{eqnarray*}}

\newcommand{\BIGOP}[1]{\mathop{\mathchoice%
		{\raise-0.22em\hbox{\huge $#1$}}%
		{\raise-0.05em\hbox{\Large $#1$}}{\hbox{\large $#1$}}{#1}}}

\newcommand{\BIGboxplus}{\mathop{\mathchoice%
		{\raise-0.35em\hbox{\huge $\boxplus$}}%
		{\raise-0.15em\hbox{\Large $\boxplus$}}{\hbox{\large $\boxplus$}}{\boxplus}}}

\begin{document}

\title{Extended Laplace Principle for Empirical Measures of a Markov Chain}

\author{Stephan Eckstein\thanks{Department of Mathematics, University of Konstanz, 78464 Konstanz Germany, stephan.eckstein@uni-konstanz.de \newline The author likes to thank Daniel Bartl, Michael Kupper and Daniel Lacker for their valuable comments, suggestions and overall help with this work.}
}

\date{\today }

\maketitle

\begin{abstract}
	We consider discrete time Markov chains with Polish state space. The large deviations principle for empirical measures of a Markov chain can equivalently be stated in Laplace principle form, which builds on the convex dual pair of relative entropy (or Kullback-Leibler divergence) and cumulant generating functional $f\mapsto \ln \int \exp(f)$. Following the approach by Lacker \cite{lacker2016non} in the i.i.d.~case, we generalize the Laplace principle to a greater class of convex dual pairs. We present in depth one application arising from this extension, which includes large deviations results and a weak law of large numbers for certain robust Markov chains - similar to Markov set chains - where we model robustness via the first Wasserstein distance. The setting and proof of the extended Laplace principle are based on the weak convergence approach to large deviations by Dupuis and Ellis \cite{dupuis2011weak}.
\end{abstract}

\setcounter{equation}{0}

\section{Introduction}
Throughout the paper, let $(E,d)$ be a Polish space, $\mathcal{P}(E)$ the space of Borel probability measures on $E$ endowed with the topology of weak convergence and $C_b(E)$ the space of continuous and bounded functions mapping $E$ into $\mathbb{R}$. Let a Markov chain with state space $E$ be given by its initial distribution $\pi_0 \in \mathcal{P}(E)$ and Borel measurable transition kernel $\pi: E \rightarrow \mathcal{P}(E)$, and denote by $\pi_n \in \mathcal{P}(E^n)$ the joint distribution of the first $n$ steps of the Markov chain.
Define the empirical measure map $L_n : E^n \rightarrow \mathcal{P}(E)$ by
\[
L_n(x_1,...,x_n) = \frac{1}{n} \sum_{i=1}^n \delta_{x_i}
\]
and recall the relative entropy $R: \mathcal{P}(E) \times \mathcal{P}(E) \rightarrow [0,\infty]$ given by
\[
R(\nu,\mu) = \int_{E} \log\left(\frac{d\nu}{d\mu}\right) d\nu, \text{ if } \nu \ll \mu,~~~R(\nu,\mu) = \infty, \text{ else.}
\]
The main goal of this paper is to generalize the large deviations result for empirical measures of a Markov chain in its Laplace principle form. Under suitable assumptions on the Markov chain, the usual Laplace principle for empirical measures of a Markov chain states that for all $F \in C_b(\mathcal{P}(E))$
\begin{align}
\label{intro1}
\lim_{n\rightarrow \infty} \frac{1}{n} \ln \int_{E^n} \exp(n F \circ L_n) d\pi_n = \sup_{\nu \in \mathcal{P}(E)} (F(\nu) - I(\nu)).
\end{align}
Here, $I : \mathcal{P}(E) \rightarrow [0,\infty]$ is the rate function, given in the setting of \cite[Chapter 8]{dupuis2011weak} by
\[
I(\nu) = \inf_{q : \nu q = \nu} \int_E R(q(x),\pi(x)) \nu(dx), 
\]
where the infimum is over all stochastic kernels $q$ on $E$ that have $\nu$ as an invariant measure.\footnote{A stochastic kernel $q$ on $E$ is a Borel measurable mapping $q: E \rightarrow \mathcal{P}(E)$. We define $\nu q \in \mathcal{P}(E)$ by $\nu q (A) := \int_E q(x,A) \nu(dx)$ for $\nu \in \mathcal{P}(E)$, where we write $q(x,A) = q(x)(A)$ for $x \in E$ and Borel sets $A \subseteq E$.} The Laplace principle \eqref{intro1} - in the mentioned setting of \cite{dupuis2011weak} - is equivalent to the more commonly used form of the large deviations result for empirical measures of a Markov chain, which states that for all Borel sets $A\subseteq \mathcal{P}(E)$
\[
-\inf_{\nu \in \mathring{A}} I(\nu) \leq \liminf \frac{1}{n} \ln\pi_n(L_n \in \mathring{A}) \leq \limsup \frac{1}{n} \ln\pi_n(L_n \in \bar{A}) \leq -\inf_{\nu \in \bar{A}} I(\nu),
\]
where $\mathring{A}$ denotes the interior and $\bar{A}$ the closure of $A$. Large deviations probabilities of Markov chains have been studied in a variety of settings and under different assumptions, see e.g.~\cite{de1990large,dembo2010large,donsker1976asymptotic,donsker1975asymptotic,jain1990large,ney1987markov}.

The way we generalize the Laplace principle is by using the fact that both sides of the Laplace principle \eqref{intro1} can be stated solely in terms of relative entropy, its chain rule, and its convex dual pair. Equation \eqref{intro1} can therefore be formulated analogously for functionals resembling the relative entropy, in the sense that these functionals have to satisfy the same type of chain rule and duality. The kind of convex duality referred to is Fenchel–Moreau duality, often studied in the context of convex risk measures, similar to our use for example in \cite{acciaio2011dynamic,bartl2016pointwise,cheridito2011composition,lacker2015law}.

The original idea for extensions of Laplace principles of this form is due to Lacker \cite{lacker2016non} who pursued this in the context of i.i.d.~sequences of random variables instead of Markov chains. The initial goal was to provide a setting to study more than just exponential tail behavior of random variables, as is given by large deviations theory. The extension of Sanov's theorem he proved \cite[Theorem 3.1]{lacker2016non} can be used to derive many interesting results, such as polynomial large deviations upper bounds, robust large deviations bounds, robust laws of large numbers, asymptotics of optimal transport problems, and more, while several possibilities remain unexplored.

In this paper, the same type of extension for Markov chains is obtained. To this end, we work in a similar setting as \cite[Chapter 8]{dupuis2011weak}. In particular, the results from \cite[Chapter 8]{dupuis2011weak} are a special case of Theorem \ref{mainth}.\footnote{Up to very minor differences with regard to the initial distribution: In this paper we work with arbitrary initial distributions, while \cite{dupuis2011weak} work with suprema over Dirac measures supported by a compact set.} To showcase the potential implications of Theorem \ref{mainth}, we focus on one broad application related to robust Markov chains, summarized in Theorem \ref{MainApplicationSum} and Theorem \ref{limitweakrobust}.

\subsection{Main Results}
Let $\beta : \mathcal{P}(E) \times \mathcal{P}(E) \rightarrow (-\infty,\infty]$ 
be a Borel measurable function which is bounded from below and satisfies $\beta(\nu,\nu) = 0$ for all $\nu \in \mathcal{P}(E)$. One may think of $\beta(\cdot,\cdot) = R(\cdot,\cdot)$. To state the chain rule, we introduce the following notation for the decomposition of an $n$-dimensional measure $\nu \in \mathcal{P}(E^n)$ into kernels $\nu_{i,i+1} : E^i \rightarrow \mathcal{P}(E)$ for $i = 1,...,n-1$ and $\nu_{0,1} \in \mathcal{P}(E)$:
\[
\nu(dx_1,...,dx_n) = \nu_{0,1}(dx_1) \prod_{i=1}^{n-1} \nu_{i,i+1}(x_1,...,x_i,dx_{i+1})
\]
For $\theta \in \mathcal{P}(E)$, define $\beta^{\theta}_n : \mathcal{P}(E^n) 
\rightarrow 
(-\infty,\infty]$ by 
\[
\beta^{\theta}_n(\nu) = \beta(\nu_{0,1},\theta) + \int_{E^n} \sum_{i=1}^{n-1} 
\beta(\nu_{i,i+1}(x_1,...,x_{i}),\pi(x_i)) \nu(dx_1,...,dx_n),
\]
where in case of $\beta(\cdot,\cdot) = R(\cdot,\cdot)$ one gets $\beta^{\pi_0}_n(\nu) = R(\nu,\pi_n)$ for $\nu \in \mathcal{P}(E^n)$ by the chain rule for relative entropy.
Note that $\beta_n^{\cdot}(\cdot)$ is well defined as the term inside the integral is Borel measurable, e.g.~by \cite[Prop.~7.27]{bertsekas2004stochastic}.
Define $\rho^{\theta}_n$ as the convex dual of $\beta_n^{\theta}$ by 
\[
\rho^{\theta}_n(f) = \sup_{\mu \in \mathcal{P}(E^n)} \left( \int_{E^n} f d 
\mu - 
\beta^{\theta}_n(\mu)\right)
\]
for Borel measurable functions $f: E^n \rightarrow \mathbb{R}$,
where we adopt the convention $\infty - \infty := -\infty$. For $\beta(\cdot,\cdot) = R(\cdot,\cdot)$ we get $\rho_n^{\pi_0}(f) = \ln \int_{E^n} \exp(f) d\pi_n$ by the Donsker-Varadhan variational formula for the relative entropy.
In the above definitions, $\theta$ is a placeholder for variable initial distributions, which is
required as a tool in the proof. For the actual statement, only 
$\beta_n^{\pi_0}$ and $\rho_n := 
\rho^{\pi_0}_n$ are needed.
We write $\rho := \rho_1$ and $\rho^{\theta} := \rho^{\theta}_1$.

The assumptions for the main theorem are stated below. Assumption (M) is \cite[Condition 8.4.1.]{dupuis2011weak}, and (T) is a direct generalization of \cite[Condition 8.2.2.]{dupuis2011weak}.
\begin{itemize}
	\item[(M)] Conditions on the Markov chain.
	\begin{itemize}
		\item[(M.1)] Define the $k$-step transition kernel $\pi^{(k)}$ of the 
		Markov chain recursively by $\pi^{(k)}(x,A) := \int_{E}\pi(y,A) 
		\pi^{(k-1)}(x,dy)$ for $x\in E$ and Borel sets $A \subseteq E$.
		
		Assume that there exist $l_0,n_0 \in \mathbb{N}$ such that for all 
		$x,y \in E$:
		\[
		\sum_{i = l_0}^{\infty} \frac{1}{2^i} \pi^{(i)}(x) \ll \sum_{j= 
			n_0}^{\infty} \frac{1}{2^j}\pi^{(j)}(y)
		\]
		\item[(M.2)] $\pi$ has an invariant measure, i.e.~there exists $\mu^* 
		\in \mathcal{P}(E)$ such that $\mu^* \pi = \mu^*$.
	\end{itemize}
	\item[(B)] Assumptions on $\beta$.
	\begin{itemize}
		\item[(B.1)] The mapping 
		$\mathcal{P}(E)\times\mathcal{P}(E^2)\ni(\theta,\mu) \mapsto 
		\beta^{\theta}_2(\mu)$ 
		is convex.
		\item[(B.2)] The mapping 
		$\mathcal{P}(E)\times\mathcal{P}(E^2)\ni(\theta,\mu) \mapsto 
		\beta^{\theta}_2(\mu)$ 
		is lower semi-continuous.
		\item[(B.3)] If $\nu$ is not absolutely continuous with respect to 
		$\mu$, then $\beta(\nu,\mu) = \infty$.
	\end{itemize}
	\item[(T)] Assumption needed to guarantee tightness of certain families of random variables. At least one of the following has to hold:
	\begin{itemize}
		\item[(T.1)] There exists a Borel measurable function $U : E \rightarrow 
		[0,\infty)$ such that the following holds:
		\begin{itemize}
			\item[(a)] $\inf_{x\in E} (U(x) - \rho^{\pi(x)}(U)) > -\infty$.
			\item[(b)] $\{ x \in E : U(x) - 
			\rho^{\pi(x)}(U) \leq M \}$ is a relatively compact subset of 
			$E$ for all $M \in \mathbb{R}$.
			\item[(c)] $\rho(U) < \infty$.
		\end{itemize}
		\item[(T.1')] E is compact.
	\end{itemize}
\end{itemize}
In case of $\beta(\cdot,\cdot) = R(\cdot,\cdot)$, one usually imposes another condition on $\pi$ in the form of the Feller property, i.e.~continuity of $x\mapsto\pi(x)$, see e.g.~\cite[Condition 8.3.1]{dupuis2011weak}. Here, this is implicitly included in condition (B.2). Indeed, one quickly checks that for (B.2) to hold in case of $\beta(\cdot,\cdot) = R(\cdot,\cdot)$, the following is sufficient: If $\theta_n \stackrel{w}{\rightarrow} \theta \in \mathcal{P}(E)$, then $\theta_n \otimes \pi \stackrel{w}{\rightarrow} \theta \otimes \pi \in \mathcal{P}(E^2)$ has to hold as well. The Feller property implies this, see \cite[Lemma 8.3.2.]{dupuis2011weak}.

The following extension of the Laplace principle for empirical measures of a Markov chain is the main result.
\begin{theorem}
	\label{mainth}
	Define the rate function $I: \mathcal{P}(E) \rightarrow (-\infty,\infty]$ by \begin{align}
	\label{eq::RateFunction}
	I(\nu) := \inf_{q : \nu q = \nu} \int_E \beta(q(x),\pi(x)) \nu(dx) = \inf_{q : \nu q = \nu} \beta_2^{\nu}(\nu\otimes q).
	\end{align}
	
	Under condition (B.1), (B.2) and 
	(T), the upper bound
	\[
	\limsup_{n \rightarrow \infty} \frac{1}{n} \rho_n(n F\circ L_n) \leq 
	\sup_{\nu \in \mathcal{P}(E)}\left( F(\nu) - I(\nu)\right)
	\]
	holds for all upper semi-continuous and bounded functions $F : \mathcal{P}(E) \rightarrow \mathbb{R}$.
	
	Under condition  (M.1), (M.2), (B.1) and (B.3), the lower bound
	\[
	\liminf_{n \rightarrow \infty} \frac{1}{n} \rho_n(n F\circ L_n) \geq 
	\sup_{\nu \in \mathcal{P}(E)}\left( F(\nu) - I(\nu)\right)
	\]
	holds for all $F \in C_b(\mathcal{P}(E))$.
\end{theorem}
Intuition, applicability and difficulties in dealing with the above result are 
very similar to the i.i.d.~case and are described in detail in the introduction 
of \cite{lacker2016non}. The main differences for Markov chains are conditions (B.1) and (B.2). 
To verify these conditions, one would ideally like to have a 
better expression for $\beta_2^{\cdot}(\cdot)$ than is given by the definition, which is often not trivial. In the applications of this paper, the choices of $\beta$ are convenient in 
this regard. Some of the applications pursued in the i.i.d.~case, e.g.~\cite[Chapter 
4 and 6]{lacker2016non} appear more difficult to obtain for Markov chains. A 
thorough analysis of the spectrum of applications of Theorem \ref{mainth} is 
left open for now, as the goal in this regard is rather to give a detailed 
account of the applications to robust Markov chains.

The following corollary complements Theorem \ref{mainth}.

\begin{corollary}
	\label{cor1}
	\begin{itemize}
		\item[(a)] If $\beta_2^{\cdot}(\cdot)$ is lower semi-continuous, then $I$ is lower semi-continuous. If $\beta_2^{\cdot}(\cdot)$ is convex, then $I$ is convex.
		\item[(b)] 	If the main Theorem \ref{mainth} upper bound holds, and additionally $I$ has compact sub-level sets, 
		then the main theorem upper bound extends to all functions $F: \mathcal{P}(E) \rightarrow [-\infty,\infty)$ which are upper semi-continuous and bounded from above.
	\end{itemize}
\end{corollary}

\subsection{Applications to robust Markov chains}
\label{subsec::Applicationsto}
In this paper, robustness broadly refers to uncertainty about the correct model specification of the Markov chain. This type of uncertainty is often studied in terms of nonlinear expectations (see e.g.~\cite{cerreia2016ergodic,lan2017strong,peng2009survey,peng2010nonlinear}) and distributional robustness (see e.g.~\cite{blanchet2016quantifying,esfahani2015data,gao2016distributionally,hanasusanto2015distributionally}). Here, the main point is to take uncertainty with respect to the transition kernel $\pi$ into consideration. Conceptually, a robust transition kernel is the following: If the Markov chain is in point $x\in E$, the next step of the Markov chain is not necessarily determined by a fixed measure $\pi(x)$, but rather can be determined by any measure $\hat{\pi} \in P(x) \subseteq \mathcal{P}(E)$. In our context, $P(x)$ will be defined as a neighborhood of $\pi(x)$ with respect to the first Wasserstein distance.

The existing literature on robust Markov chains focuses on finite state spaces, where transition probabilities are uncertain in some convex and closed sets, usually expressed via matrix intervals. For example \cite{vskulj2009discrete} gives a good overview of the field. These are studied under the names of Markov set chains (see e.g.~\cite{hartfiel1994theory,hartfiel2006markov,kurano1998controlled}), imprecise Markov chains (see e.g.~\cite{de2009imprecise}), as well as Markov chains with interval probabilities (see e.g.~\cite{vskulj2006finite,vskulj2009discrete}). While different types of laws of large numbers are studied frequently, large deviations theory seems to be absent in the current literature on robust Markov chains.

In the following, the asymptotic behavior of such Markov chains is analyzed. The type of asymptotics studied are worst case behaviors over all possible distributions, in the sense of large deviations probabilities (Theorem \ref{MainApplicationSum}) and a law of large numbers (Theorem \ref{limitweakrobust}) of empirical measures of robust Markov chains. Worst case behavior for large deviations means that the slowest possible rate of convergence to zero of a tail event is identified. For laws of large numbers, we give upper bounds - or by changing signs lower bounds - for law of large number type limits.

Define the first Wasserstein distance $d_W$ on $\mathcal{P}(E)$ by
\[d_W(\mu,\nu) = \inf_{\tau \in \Pi(\mu,\nu)} \int_E d(x,y)\tau(dx,dy)\]
for $\mu,\nu \in \mathcal{P}(E)$, where $\Pi(\mu,\nu) \subseteq \mathcal{P}(E^2)$ denotes the set of measures with first marginal $\mu$ and second marginal $\nu$.
See for example \cite{gibbs2002choosing} for an overview regarding the Wasserstein distance. In order to avoid complications with respect to compatibility of weak convergence and Wasserstein distance, we assume that $E$ is compact for the applications.

Fix $r \geq 0$. The set of possible joint distributions of the robust Markov chain up to step $n$ is characterized by $M_n(\pi_0) \subseteq \mathcal{P}(E^n)$ defined by
\begin{align*}
M_n(\pi_0) :=& \{\nu \in \mathcal{P}(E^n) : d_W(\nu_{0,1},\pi_0)\leq r \text{ and } d_W(\nu_{i,i+1}(x_1,...,x_i),\pi(x_i)) \leq r
~\nu\text{-a.s. for } 
i=1,...,n-1\}.
\end{align*}
For technical reasons related to condition (B.3), we also consider the following modification
\begin{align*}
\underline{M}_n(\pi_0) :=& \{ \nu \in M_n(\pi_0) : \nu \ll \pi_0\otimes \pi \otimes ... \otimes \pi \}.
\end{align*}
Both definitions above can of course be stated for arbitrary $\theta \in \mathcal{P}(E)$ instead of $\pi_0$.
We show that
\[
\beta(\nu,\mu) := \inf_{\hat{\mu} \in M_1(\mu)} R(\nu,\hat{\mu})
\]
satisfies the assumptions for the upper bound of Theorem \ref{mainth} and
\[
\underline{\beta}(\nu,\mu) := \inf_{\hat{\mu} \in \underline{M}_1(\mu)} R(\nu,\hat{\mu})
\]
satisfies the assumptions for the lower bound of Theorem \ref{mainth}. In Lemma \ref{robustentropyselection} and Lemma \ref{robustentropyselection2} we will characterize $\beta_n^\theta$ and $\underline{\beta}_n^\theta$ in terms of $M_n(\theta)$ and $\underline{M}_n(\theta)$.
Theorem \ref{mainth} yields the following:
\begin{theorem}
	\label{MainApplicationSum}
	Assume $(E,d)$ is compact. Let $\beta$, $\underline{\beta} $ and $M_n(\theta)$, $\underline{M}_n(\theta)$ for $\theta \in \mathcal{P}(E)$ be given as above. Let $I$ and $\underline{I}$ denote the rate functions for $\beta$ and $\underline{\beta}$, as given by equation \eqref{eq::RateFunction}.
	
	\begin{itemize}
		\item[(a)] If $\pi$ satisfies the Feller property, it holds for Borel sets $A \subseteq \mathcal{P}(E)$
		\[
		\limsup_{n\rightarrow \infty} \sup_{\mu \in M_n(\pi_0)} \frac{1}{n} \ln \mu(L_n \in \bar{A}) \leq -\inf_{\nu \in \bar{A}} I(\nu).
		\]
		\item[(b)] If $\pi$ satisfies (M), it holds for Borel sets $A \subseteq \mathcal{P}(E)$
		\[\liminf_{n\rightarrow \infty} \sup_{\mu \in \underline{M}_n(\pi_0)} \frac{1}{n} \ln \mu(L_n \in \mathring{A}) \geq -\inf_{\nu \in \mathring{A}} \underline{I}(\nu).\]
	\end{itemize}
\end{theorem}
For a (numerical) illustration of the above result, see Example \ref{simpleexample}. Among other things, the example showcases that often, there is no difference between upper and lower bound, and thus the above identifies precise asymptotic rates. Note that in finite state spaces one can guarantee $M_n(\theta) = \underline{M}_n(\theta)$ by assuming $\pi(x)(y) > 0$ for all $x,y\in E$. 

The following is the law of large numbers result for robust Markov chains, which is based on the choices
\begin{align*}
\beta(\mu,\nu) &:= \left\{ \begin{array}{ll} 0,& \text{if } d_W(\mu,\nu)\leq r, \\ 
\infty,&
\text{else,} \end{array} \right.\\
\underline{\beta}(\mu,\nu) &:= \left\{ \begin{array}{ll} 0,& \text{if } d_W(\mu,\nu)\leq r \text{ and } \mu \ll \nu, \\ 
\infty,&
\text{else,} \end{array} \right.
\end{align*}
again for $r \geq 0$ fix.
\begin{theorem}
	\label{limitweakrobust}
	Assume $(E,d)$ is compact. Let $M_n(\theta),\underline{M}_n(\theta)$ for $\theta \in \mathcal{P}(E)$ be given as above.
	\begin{itemize}
		\item[(a)] If $\pi$ satisfies the Feller property, it holds for all $F: \mathcal{P}(E) \rightarrow [-\infty,\infty)$ which are upper semi-continuous and bounded from above
		\[
		\limsup_{n\rightarrow\infty} \sup_{\mu \in M_n(\pi_0)} \int_{E^n} F\circ L_n d\mu \leq \sup_{\substack{\nu \in \mathcal{P}(E) : \\ \exists q, \nu q = \nu: \nu \otimes q \in M_2(\nu)}} F(\nu).
		\]
		\item[(b)] If $\pi$ satisfies (M), it holds for all $F \in C_b(\mathcal{P}(E))$
		\[
		\liminf_{n\rightarrow \infty} \sup_{\mu \in \underline{M}_n(\pi_0)} \int_{E^n} F \circ L_n d\mu \geq \sup_{\substack{\nu \in \mathcal{P}(E) : \\ \exists q, \nu q = \nu: \nu \otimes q \in \underline{M}_2(\nu)}} F(\nu).
		\]
	\end{itemize}
\end{theorem}
This result is easiest interpreted by looking at the case $r = 0$. If both upper and lower bound hold, the above states
\[
\pi_n \circ L_n^{-1} \stackrel{w}{\rightarrow} \delta_{\mu^*} \in \mathcal{P}(\mathcal{P}(E)),
\]
where $\mu^*$ is the unique invariant measure under the Markov chain transition kernel $\pi$, which - under condition (M) - always exists.

Specifically, the choices $F(\nu) := \int_E f d\nu$ for $f \in C_b(E)$ in the 
above can be interpreted as a robust Ces\`aro limit of a Markov chain. Indeed, 
for $r = 0$, this yields
\[
\lim_{n\rightarrow \infty}\frac{1}{n} \sum_{i=1}^{n} \pi_0 \pi^{(i-1)} \stackrel{w}{\rightarrow} \mu^*.
\]
For $r > 0$ however, we get a result which strongly resembles e.g.~\cite[Theorem 4.1]{hartfiel1994theory}, but in a more general state space. 

\subsubsection{Generalizations and relation to the literature}
In this paper robustness is modeled via the first Wasserstein distance because it is both tractable and frequently used. Nevertheless, the question arises whether the presented approach can be applied more generally, specifically related to the existing literature in finite state spaces. This section roughly outlines potential extensions.

In the existing literature regarding robust Markov chains in finite state spaces - where we mainly refer to \cite{hartfiel1994theory,vskulj2009discrete} as references - the starting point is a robust transition kernel $P : E \rightarrow 2^{\mathcal{P}(E)}$ satisfying certain convexity and closedness conditions. For our approach however, one starts with both a transition kernel $\pi : E \rightarrow \mathcal{P}(E)$ and a mapping $U: \mathcal{P}(E) \rightarrow 2^{\mathcal{P}(E)}$, with the relation of the approaches being $P = U \circ \pi$.

In the previous Section \ref{subsec::Applicationsto} we used $U(\mu) = 
\{\hat{\mu} \in \mathcal{P}(E) : d_W(\mu,\hat{\mu})\leq r\}$.\footnote{The 
setting of Section \ref{subsec::Applicationsto} translates to $\beta(\nu,\mu) = 
\inf_{\hat{\mu}\in U(\mu)} R(\nu,\hat{\mu})$ for large deviations results 
(Theorem \ref{MainApplicationSum}) and $\beta(\nu,\mu) = \infty \cdot 
\eins_{U(\mu)^C}(\nu)$ for law of large numbers results (Theorem 
\ref{limitweakrobust}). Further $M_n(\theta) = \{\mu \in \mathcal{P}(E^n) : 
\mu_{0,1} \in U(\theta),~ \mu_{i,i+1}(x_1,...,x_i) \in 
U(\pi(x_i))~\mu\text{-a.s.~for }i=1,...,n-1\}$ for $\theta \in 
\mathcal{P}(E)$.} In general, the following conditions on $U$ would allow for a 
similar type of proof of analogs of Theorem \ref{MainApplicationSum} and 
Theorem \ref{limitweakrobust}, where the assumptions on $E$ (compactness) and 
$\pi$ (Feller property and/or (M)) stay the same.
\begin{itemize}
	\item[(1)] $\mu \in U(\mu)$ for all $\mu \in \mathcal{P}(E)$.
	\item[(2)] The graph of $U$, i.e.~$\{(\mu,\hat{\mu}) \in \mathcal{P}(E)^2 : 
	\hat{\mu} \in U(\mu)\}$, is closed and convex.
\end{itemize}
Here, (1) implies $\beta(\mu,\mu) = 0$ for all $\mu \in \mathcal{P}(E)$. That 
the graph of $U$ is convex implies condition (B.1), see Lemma \ref{Mnconv} and 
the subsequent paragraph, as well as Lemma \ref{weaklimbeta}. Closedness of the 
graph is used to verify condition (B.2), see Lemma 
\ref{M2closed}, \ref{lscLDP} and \ref{weaklimbeta}. For the large deviations 
result, closedness of the graph also guarantees a representation of 
$\beta_n^{\theta}$ in terms of $M_n(\theta)$, see Lemma 
\ref{robustentropyselection} and \ref{robustentropyselection2}. 

The assumption 
that $E$ has to be compact can likely be loosened by assuming that $U$ is 
compact valued instead, even though an analog of Lemma \ref{M2closed} is then 
more difficult to obtain.

\subsection{Structure of the paper}
In the following Section \ref{sec::mainth}, we prove Theorem \ref{mainth} and Corollary \ref{cor1}. The method of proof is oriented at \cite[Chapter 8 and 9]{dupuis2011weak}, while also using tools from convex duality and measurable selection.
Section \ref{subsec::lowerbound} gives results related to and the proof of the lower bound, Section \ref{subsec::UpperBound} results related to and the proof of the upper bound, and Section \ref{subsec::Cor} the proof of Corollary \ref{cor1}.

In Section \ref{sec::Appl}, we present in depth the applications to robust Markov chains. Aside from using Theorem \ref{mainth} and Corollary \ref{cor1}, Section \ref{sec::Appl} is self-contained, so readers who prefer to read Section \ref{sec::Appl} before Section \ref{sec::mainth} can easily do so. A large part of Section \ref{sec::Appl} is devoted to verify conditions (B.1) and (B.2) for the different choices of $\beta$. Further, the obtained large deviations results are illustrated in Example \ref{simpleexample}.

Many of the smaller results not listed in the introduction are interesting for their own sake, e.g.~Lemma \ref{ughselection}, Lemma \ref{allpi0lower} and Lemma \ref{M2closed}.

\section{Proof of Theorem \ref{mainth} and Corollary \ref{cor1}}
\label{sec::mainth}
\subsection{Main Theorem Lower Bound}
\label{subsec::lowerbound}
In this section, at some points it is necessary to evaluate $\rho^{\theta}_n$ at universally measurable functions, which is still well defined. More precisely, upper semi-analytic functions are the object of interest, the reason made obvious in Lemma \ref{ughselection}. In particular, upper semi-analytic functions are universally measurable. See e.g.~\cite[Chapter 7]{bertsekas2004stochastic} for background. 
\subsubsection{Preliminary Results}
\label{subsubsec::lowerboundprelim}
\begin{lemma}{(See also \cite[Prop.~A.1]{lacker2016non})}
	\label{ughselection}
	For $\theta \in \mathcal{P}(E)$, $f\in C_b(E^n)$ and $0 < k < n$ it holds
	\[
	\rho_n^{\theta}(f) = \rho_k^{\theta}(g),
	\]
	where $g: E^k \rightarrow \mathbb{R}$ is defined by
	\[
	g(x_1,...,x_k) = \rho_{n-k}^{\pi(x_k)}(f(x_1,...,x_k,\cdot)).
	\]
	Further, $g$ is upper semi-analytic.
	\begin{proof}
		First, let $\nu \in \mathcal{P}(E^k)$ and $K: E^k \rightarrow \mathcal{P}(E^{n-k})$ be a stochastic kernel.
		For notational purposes, we write $\bar{x} = (x_1,...,x_k)$ for $x_1,...,x_k \in E$ and
		\[
		K(x_1,...,x_k) = K(\bar{x}) = K^{\bar{x}}.
		\]
		Denote the decomposition of $K^{\bar{x}}$ in the usual way
		\[
		K^{\bar{x}} = K^{\bar{x}}_{0,1} \otimes K^{\bar{x}}_{1,2} \otimes ... \otimes K^{\bar{x}}_{n-k-1,n-k}.
		\]
		For the decompositions of $\nu$ and $\nu\otimes K$ the trivial $\nu\otimes K$-almost sure equalities hold
		\begin{align*}
		\nu_{i,i+1}(x_1,...,x_i) &= (\nu \otimes K)_{i,i+1}(x_1,...,x_i) &\text{ for } i = 0,...,k-1,\\
		K^{\bar{x}}_{i,i+1}(x_{k+1},...,x_{k+i}) &= (\nu \otimes K)_{k+i,k+i+1}(x_1,...,x_{k+i}) &\text{ for } i = 0,...,n-k-1.
		\end{align*}
		Hence
		\begin{align*}
		&\beta_k^{\theta}\left(\nu\right)  + \int_{E^k} \beta_{n-k}^{\pi\left(x_k\right)}\left(K^{\bar{x}}\right) \nu\left(dx_1,...,dx_k\right)\\
		&= \int_{E^n} \beta\left(\nu_{0,1},\theta\right) + \left(\sum_{i=1}^{k-1} \beta\left(\nu_{i,i+1}\left(x_1,...,x_i\right),\pi\left(x_i\right)\right)\right) + \beta\left(K^{\bar{x}}_{0,1},\pi\left(x_k\right)\right)\\ &+ \left(\sum_{i=1}^{n-k-1} \beta\left(K^{\bar{x}}_{i,i+1}\left(x_{k+1},...,x_{k+i}\right),\pi\left(x_{k+i}\right)\right)\right) K^{\bar{x}}\left(dx_{k+1},...,dx_n\right) \nu\left(dx_1,...,dx_k\right)\\
		&= \beta_n^{\theta}\left(\nu\otimes K\right).
		\end{align*}
		Using the above and a standard measurable selection argument \cite[Proposition 7.50]{bertsekas2004stochastic} we get
		\begin{align*}
		&\rho_k^{\theta}(g)\\
		&= \sup_{\nu\in \mathcal{P}(E^k)} \left( \int_{E^k} g d\nu - \beta_k^{\theta}(\nu)\right)\\
		&= \sup_{\nu \in \mathcal{P}(E^k)} \left( \int_{E^k} \sup_{\mu \in \mathcal{P}(E^{n-k})} \left( \int_{E^{n-k}} f(x_1,...,x_n) \mu(dx_{k+1},...,dx_n) - \beta_{n-k}^{\pi(x_k)}(\mu)\right) \nu(dx_1,...,dx_k) - \beta_k^{\theta}(\nu)\right)\\
		&= \sup_{\nu \in \mathcal{P}(E^k)}\sup_{\substack{K: E^k \rightarrow \mathcal{P}(E^{n-k}),\\K\text{ Borel}}} \left( \int_{E^n} f d\nu \otimes K - \beta_n^{\nu}(\nu \otimes K)\right)\\
		&= \rho_n^{\theta}(f).
		\end{align*}
		
		That $g$ is upper semi-analytic can be shown as follows: 
		Both mappings
		\begin{align*}
		(x_k,\nu) \mapsto -\beta_{n-k}^{\pi(x_k)}(\nu),\\
		(x_1,...,x_k,\nu) \mapsto \int_{E^{n-k}} f(x_1,...,x_k,\cdot) d\nu
		\end{align*}
		are upper semi-analytic by \cite[Prop.~7.48]{bertsekas2004stochastic}, where for the first mapping we implicitly have to use \cite[Prop 7.27]{bertsekas2004stochastic} as mentioned after the definition of $\beta_n^{\cdot}(\cdot)$.
		The sum of these mappings is therefore still upper semi-analytic (see e.g.~\cite[Lemma 7.30 (4)]{bertsekas2004stochastic}) and hence by \cite[Prop.~7.47]{bertsekas2004stochastic} we get that $g$ is upper semi-analytic.
	\end{proof}
\end{lemma}
\begin{lemma}
	\label{allpi0lower}
	Under condition (B.3), for all $\theta \in \mathcal{P}(E)$ and $f\in C_b(E^n)$ it holds
	\[
	\rho_n^{\theta}(f) \geq \int_E \rho_n^{\delta_x}(f) \theta(dx).
	\]
	\begin{proof}
		Let $f\in C_b(E^n)$. By condition (B.3), it holds for all $x\in E$
		\begin{align*}
		\rho_n^{\delta_x}(f) &= \sup_{\nu \in \mathcal{P}(E^n)} \left( \int_{E^n} f d\nu - \beta_n^{\delta_x}(\nu)\right)\\
		&= \sup_{\substack{\nu \in \mathcal{P}(E^n) :\\ \nu_{0,1} = \delta_x}} \left( \int_{E^n} f d\nu - \beta_n^{\delta_x}(\nu)\right)\\
		&= \sup_{\nu \in \mathcal{P}(E^{n-1})} \left( \int_{E^n} f d(\delta_x \otimes \nu) - \beta_n^{\delta_x}(\delta_x \otimes \nu)\right).
		\end{align*}
		Hence we get for $\theta \in \mathcal{P}(E)$
		\begin{align*}
		&\int_E \rho_n^{\delta_{x_1}}(f) \theta(dx_1) \\&= \int_E \sup_{\nu \in \mathcal{P}(E^{n-1})} \left( \int_{E^n} f d(\delta_{x_1} \otimes \nu) - \beta_n^{\delta_{x_1}}(\delta_{x_1} \otimes \nu)\right) \theta(dx_1)\\
		&= \int_E \sup_{\nu \in \mathcal{P}(E^{n-1})} \left( \int_{E^{n-1}} f(x_1,\cdot) d\nu - \int_{E^{n-1}} \sum_{k=2}^n \beta(\nu_{k-2,k-1}(x_2,...,x_{k-1}),\pi(x_{k-1}) \nu(dx_2,...,dx_n)\right) \theta(dx_1)	\\
		&\stackrel{(*)}{=} \sup_{\substack{K : E \rightarrow \mathcal{P}(E^{n-1})\\ K \text{Borel}}}\left( \int_{E^n} f d\theta \otimes K - \beta_n^{\theta}(\theta \otimes K)\right)\\
		&\leq \sup_{\nu \in \mathcal{P}(E^n)} \left( \int_{E^n} f d\nu - \beta_n^{\theta}(\nu)\right)\\ &= \rho_n^{\theta}(f).	
		\end{align*}
		Here, $(*)$ follows by a standard measurable selection argument, e.g.~\cite[Proposition 7.50]{bertsekas2004stochastic}.
	\end{proof}
\end{lemma}
\begin{lemma}
	\label{weakcon}
	Let $(X_i)_{i\in\mathbb{N}}$ be an $E$-valued sequence of random variables such that $\lim_{n\rightarrow 
		\infty}\frac{1}{n} \sum_{i= 1}^{n} F(X_i) = \mathbb{E}[F(X_1)]$ holds almost surely for all 
	$F\in C_b(E)$. Let $\nu^{(n)} = \mathbb{P}\circ 
	(X_1,...,X_n)^{-1}$ be the distribution of $(X_1,...,X_n)$ for 
	$n\in\mathbb{N}$. Then $\nu^{(n)} \circ L_n^{-1} \stackrel{w}{\rightarrow} 
	\delta_{\nu^{(1)}}$.
\end{lemma}
\begin{proof}
	$(E,d)$ admits an equivalent metric $m$ such that the space of uniformly continuous and bounded functions with respect to this metric
	$\mathcal{U}_b(E,m)$ is separable with respect to the uniform metric, see e.g.~Lemma 3.1.4 in \cite{stroock1993probability}.\footnote{Two metrics are equivalent if they generate the same topology. The uniform metric $\hat{m}$ on $\mathcal{U}_b(E,m)$ is given by $\hat{m}(f_1,f_2) := \sup_{x\in E} \left|f_1(x)-f_2(x)\right|$.}
	Choose a countable, 
	dense subset $A\subseteq \mathcal{U}_b(E,m)$. By assumption and since $A$ 
	is countable, we can choose a null set $N\subseteq \Omega$ such that for all 
	$\omega \in N^C$
	\[ \forall F \in A :\lim_{n\rightarrow \infty} \int_E F 
	dL_n(X_1(\omega),...,X_n(\omega)) = \lim_{n\rightarrow \infty}\frac{1}{n} 
	\sum_{i=1}^n F(X_i(\omega)) = \mathbb{E}[F(X_1)] = \int_E F d\nu^{(1)}.
	\]
	Let $F\in \mathcal{U}_b(E,m)$ and choose $(F_i)_{i\in\mathbb{N}} \subseteq A$ 
	such that $\lim_{i\rightarrow \infty}\sup_{x\in E} |F_i(x)-F(x)| = 0$.
	For all $i,n\in \mathbb{N},\omega \in N^C$, it holds
	\begin{align*}
	&\left\lvert \int_{E} F d L_n(X_1(\omega),...,X_n(\omega)) - \int_{E} F d 
	\nu^{(1)}\right\rvert \\
	\leq &\left\lvert\int_{E} (F-F_i) d 
	L_n(X_1(\omega),...,X_n(\omega))\right\rvert + \left\lvert\int_{E} (F_i -F) 
	d \nu^{(1)}\right\rvert \\ +&\left\lvert\int_{E} F_i d 
	L_n(X_1(\omega),...,X_n(\omega)) - \int_{E} F_i d \nu^{(1)}\right\rvert.
	\end{align*}
	This yields for all $\omega \in N^C$
	\[
	\lim_{n\rightarrow \infty}\int_{E} F d L_n(X_1(\omega),...,X_n(\omega)) = 
	\int_{E} F d \nu^{(1)}.
	\]
	So
	$L_n(X_1,...,X_n) \stackrel{w}{\rightarrow} \nu^{(1)}$ holds 
	$\mathbb{P}$-a.s.\footnote{Note that we first get weak convergence with respect to the equivalent metric $m$. But since weak convergence under equivalent metrics is the same, this carries over to the metric $d$.}
	Hence, for $f \in C_b(\mathcal{P}(E))$ it holds $f(L_n(X_1,...X_n)) 
	\rightarrow f(\nu^{(1)})$ $\mathbb{P}$-a.s.~by continuity of $f$ and thus by 
	dominated convergence
	\begin{align*}
	\int_{\mathcal{P}(E)} f d (\nu^{(n)}\circ L_n^{-1}) &= \int_{E^n} f(L_n) d 
	\nu^{(n)} \\
	&= \int_{\Omega} f(L_n(X_1,...,X_n)) d \mathbb{P} \\
	&\rightarrow \int_{\Omega} f(\nu^{(1)}) d\mathbb{P} \\&= f(\nu^{(1)}) = 
	\int_{\mathcal{P}(E)} f d \delta_{\nu^{(1)}}.
	\end{align*}
\end{proof}
For the following results, note that under condition (M), $\pi$ has a unique invariant measure, which we denote by $\mu^*$, see Lemma 8.6.2.~(a) of \cite{dupuis2011weak}.
\begin{lemma}[Lemma 8.6.2.~(b) of \cite{dupuis2011weak}]
	\label{862b}
	Let (M) be satisfied.
	Let $A \subseteq E$ be a Borel set such that $\pi^{(l_0)}(x_0,A) > 0$ for some $x_0 \in E$. Then $\mu^*(A) > 0$, where $\mu^*$ is the unique invariant measure under $\pi$.
\end{lemma}

\begin{lemma}[Adapted version of Lemma 8.6.2.~(c) of \cite{dupuis2011weak}]
	\label{862c}
	Let (M) and (B.3) be satisfied.
	Let $\nu \in \mathcal{P}(E)$ satisfy $\beta_2^{\nu}(\nu\otimes p) <\infty$ for some stochastic kernel $p$ on $E$ such that $\nu p = \nu$. Then it holds $\nu \ll \mu^*$, where $\mu^*$ is the unique invariant measure under $\pi$.
	\begin{proof}
		Let $\Omega_0 \subseteq E$ be a Borel set such that $\nu(\Omega_0) = 1$ and $p(x) \ll \pi(x)$ for all $x \in \Omega_0$, which we can choose by (B.3) and since $\beta_2^{\nu}(\nu\otimes p) < \infty$.
		Define $\tilde{p}(x) := \eins_{\Omega_0}(x) p(x) + \eins_{\Omega_0^C}(x) \pi(x)$.
		Since $\tilde{p}(x) \ll \pi(x)$ for all $x\in E$, we have $\tilde{p}^{(l_0)}(x) \ll \pi^{(l_0)}(x)$ for all $x\in E$, where $l_0$ is the constant from condition (M.1). 
		
		Now choose a Borel set $A\subseteq E$ such that $\nu(A) > 0$. By iterating $\nu \tilde{p} = \nu$, we get a Borel set $B \subseteq E$ with $\nu(B) > 0$ and $\tilde{p}^{(l_0)}(x,A) > 0$ for all $x\in B$. Hence $\pi^{(l_0)}(x,A) > 0$ for all $x\in B$ and by Lemma $\ref{862b}$ therefore $\mu^*(A) > 0$.
	\end{proof}
\end{lemma}

\subsubsection{Proof of Theorem \ref{mainth} Lower Bound}
Let $F\in C_b(\mathcal{P}(E))$ and $\varepsilon > 0$ be fix. We have to show 
\begin{align}
\label{lowerboundmain}
\liminf_{n\rightarrow \infty} \frac{1}{n} \rho_n(n F \circ L_n) \geq \sup_{\nu \in \mathcal{P}(E)} \left( F(\nu) - I(\nu)\right) - 4 \varepsilon.
\end{align}
We do this by showing every subsequence has a further subsequence which satisfies this inequality. So we fix a subsequence and relabel it by $n\in \mathbb{N}$. Labeling subsequences by the same index as the original sequence will be a common practice throughout the remainder of the paper.

\textbf{Outline of the proof:} 

First, we show that there exists a Borel set $\Phi \subseteq E$ such that $\pi^{(l_0)}(y,\Phi) = 1$ for all $y\in E$, and for all $x\in \Phi$ it holds
\begin{align}
\label{ineqlb1}
\liminf_{n\rightarrow \infty} \frac{1}{n} \rho_{n-l_0}^{\delta_x}(n F\circ L_n(x_1,...,x_{l_0})) \geq \sup_{\nu \in \mathcal{P}(E)} (F(\nu)-I(\nu)) - 3\varepsilon
\end{align}
for all $x_1,...,x_{l_0} \in E$ and a further subsequence (the same subsequence for all $x_1,...,x_{l_0}$). This subsequence then remains fix for the rest of the proof and is again labeled by $n\in\mathbb{N}$.

The next step is to use Lemma \ref{ughselection}, i.e.~for all $f\in C_b(E^n)$
\[
\rho_n(f) = \rho_{l_0}((x_1,...,x_{l_0})\mapsto \rho_{n-{l_0}}^{\pi(x_{l_0})}(f(x_1,...,x_{l_0},\cdot)),
\]
where $l_0$ is the constant from condition (M.1). This is used together with
Lemma \ref{allpi0lower}, i.e.~for all $f\in C_b(E^n)$ and $\theta \in \mathcal{P}(E)$
\[
\rho_n^{\theta}(f) \geq \int_E \rho_n^{\delta_x}(f) \theta(dx).
\]
We then use these two results to show
\begin{align}
\label{ineqlb2}
\rho_n(nF\circ L_n) \geq \rho_{l_0}(g_n),
\end{align}
where \[g_n(x_1,...,x_{l_0}) = \int_{\Phi} \rho_{n-l_0}^{\delta_x}(nF \circ L_n(x_1,...,x_{l_0},\cdot)) \pi^{(l_0)}(x_{l_0},dx).\]
We conclude by combining the first limit result \eqref{ineqlb1} and inequality \eqref{ineqlb2}, which works by Fatou's Lemma, using monotonicity of $\rho_n$ and the fact that $\rho_n(c) \geq c$ for all $c\in \mathbb{R}$.

\textbf{First Step:}
We show \eqref{ineqlb1} for all $x\in \Phi$ and $x_1,...,x_{l_0} \in E$, where $\Phi$ and the required further subsequence is specified later.

We can without loss of generality choose $\nu_0 \in \mathcal{P}(E)$ such that
\[
-\infty < \sup_{\nu \in \mathcal{P}(E)}\left( F(\nu) - I(\nu)\right) \leq F(\nu_0)-I(\nu_0) + \varepsilon < \infty,
\]
since if the supremum equals $-\infty$, there is nothing to show.
Then
\[\inf_{q: \nu_0 q = \nu_0} 
\int_E \beta(q(x),\pi(x)) \nu_0(dx) = I(\nu_0) < \infty. \]	
Choose a stochastic 
kernel $p$ on $E$ such that
\[
\infty > I(\nu_0) + \varepsilon \geq \int_E \beta(p(x),\pi(x)) 
\nu_0(dx) = \beta_2^{\nu_0}(\nu_0 \otimes p).
\]
By (B.3), we can choose a Borel set $N \subseteq E$ with $\nu_0(N) = 0$ 
such that $p(x) \ll \pi(x)$ for all $x\in N^C$.
Define the stochastic kernel $p_0$ on $E$ by $p_0(x) := 
\eins_{N}(x) \pi(x) + \eins_{N^C}(x) p(x)$ for $x\in E$ 
and find that
\[
\infty > I(\nu_0) + \varepsilon \geq \beta_2^{\nu_0}(\nu_0 \otimes p) = \beta_2^{\nu_0}(\nu_0 \otimes p_0).
\]	
It holds $p_0(x) \ll \pi(x)$ for all $x\in E$. Next, we will replace $\nu_0$ and $p_0$ by $\nu_1$ and $p_1$, such that $F(\nu_1) + \beta_2^{\nu_1}(\nu_1 \otimes p_1) \geq 
F(\nu_0) + \beta_2^{\nu_0}(\nu_0 \otimes p_0) - 
2\varepsilon$ and additionally $p_1$ is point-wise equivalent to $\pi$.

By Condition (M.1) and (M.2), $\pi$ has a unique invariant measure,
denoted by $\mu^*$ (See Lemma 8.6.2.~(a) of \cite{dupuis2011weak}).
By lower boundedness of $\beta$ we can choose $\kappa_0 \in (0,1)$ such 
that \[ (1-\kappa_0)\beta_2^{\nu_0}(\nu_0\otimes p_0) \leq 
\beta_2^{\nu_0}(\nu_0\otimes p_0)+\varepsilon.\]
By continuity of $F$, we can further choose $\kappa_1 > 0$ such that for all $0 \leq  \hat{\kappa} \leq \kappa_1$
\[F((1-\hat{\kappa})\nu_0 + \hat{\kappa} \mu^*) \geq F(\nu_0)-\varepsilon.\]

Choose $\kappa := \min \{\kappa_1,\kappa_2\}$ and define $\nu_1 := 
(1-\kappa) \nu_0 + \kappa \mu^*$ and 
\[
p_1(x) = \frac{d\nu_0}{d\nu_1}(x) (1-\kappa) p_0(x) + \frac{d\mu^*}{d\nu_1}(x) \kappa \pi(x).
\] 
Then one quickly checks $\nu_1 \otimes p_1 = (1-\kappa) (\nu_0 \otimes p_0) + \kappa (\mu^* \otimes \pi)$. By convexity of $\beta_2^{\cdot}(\cdot)$
\begin{align*}
\beta_2^{\nu_1}(\nu_1\otimes p_1) \leq 
(1-\kappa)\beta_2^{\nu_0}(\nu_0\otimes p_0) + \kappa 
\beta_2^{\mu^*}(\mu^*\otimes \pi) \leq \beta_2^{\nu_0}(\nu_0\otimes 
p_0) + \varepsilon.
\end{align*}
and thus
\[F(\nu_1) + \beta_2^{\nu_1}(\nu_1 \otimes p_1) \geq 
F(\nu_0) + \beta_2^{\nu_0}(\nu_0 \otimes p_0) - 
2\varepsilon.\]
Since $\beta_2^{\nu_1}(\nu_1 \otimes p_1) < \infty$, without loss of generality $p_1(x) \ll \pi(x)$ for all $x \in E$. By Lemma \ref{862c} (which yields $\nu_1 \ll \mu^*$, and hence $\frac{d\mu^*}{d\nu_1}(x) > 0$ for $\nu_1$-almost all $x\in E$) and by construction of $p_1$, it also holds $\pi(x) \ll p_1(x)$, again without loss of generality for all $x\in E$.

So $p_1$ also satisfies (M.1), as every kernel which is point-wise equivalent to $\pi$ satisfies (M.1), notably with the same constants $l_0$ and $n_0$.

It follows that the Markov chain with initial distribution $\nu_1$ and 
transition kernel $p_1$ is ergodic (See Lemma 8.6.2.~(a) of 
\cite{dupuis2011weak}). The 
point-wise Ergodic Theorem\footnote{\label{footergodic}On both Ergodic Theorems used, see Appendix A.4 of \cite{dupuis2011weak} or references therein, i.e.~\cite[Corollaries 6.23 and 6.25.]{breiman1992probability}} yields that the sequence 
\[(\mu^{(n)})_{n\in\mathbb{N}} := (\nu_1 \otimes 
\left(\bigotimes_{i=1}^{n-1} p_1\right))_{n\in\mathbb{N}}\] satisfies 
the conditions for Lemma $\ref{weakcon}$ and thus $\mu^{(n)}\circ 
L_n^{-1} \stackrel{w}{\rightarrow} \delta_{\nu_1}$. This yields 
\begin{align}
\label{eqlim1}
\lim_{n\rightarrow \infty} \int_{E^n} \left| F\circ L_n - 
F(\nu_1)\right| d\mu^{(n)} = \lim_{n\rightarrow \infty} 
\int_{\mathcal{P}(E)} \left| F - F(\nu_1)\right| d\mu^{(n)} \circ 
L_n^{-1} = 0.
\end{align}

Let $(X_n)_{n\in\mathbb{N}}$ be a sequence of $E$-valued random variables such that $(X_1,...,X_n) \sim \mu^{(n)}$ for all $n\in\mathbb{N}$.
We see 
\begin{align*}
\mathbb{E}\left[\beta(p_1(X_1),\pi(X_1))\right] &= 
\beta_2^{\nu_1}(\nu_1 \otimes p_1), \\
\mathbb{E}\left[\left| 
\beta(p_1(X_1),\pi(X_1))\right|\right] &\leq \left| 
\min_{x\in \mathcal{P}(E)^2} \beta(x)\right| + \beta_2^{\nu_1}(\nu_1 
\otimes p_1) < \infty
\end{align*}
and thus by the $L_1$-ergodic theorem\footnotemark[6]:
\begin{align}
\label{eqlim2}
\lim_{n\rightarrow \infty} \mathbb{E}\left[ \left| \frac{1}{n} 
\sum_{i=1}^{n-1} 
\beta(p_1(X_i),\pi(X_i))-\beta_2^{\nu_1}(\nu_1 
\otimes p_1)\right|\right] &= 0 \\
\label{eqlim3}
\Leftrightarrow
\lim_{n\rightarrow \infty}\int_{E^n} \left| \frac{1}{n} \sum_{i=1}^{n-1} 
\beta(p_1(x_i),\pi(x_i)) - \beta_2^{\nu_1}(\nu_1 \otimes 
p_1)\right| \mu^{(n)}(dx_1,...,dx_n) &= 0.
\end{align}
For $\theta \in \mathcal{P}(E)$ and a stochastic kernel $q:E\rightarrow \mathcal{P}(E)$ we define
\begin{align*}
(\mu^{(\theta,q,n)})_{n\in\mathbb{N}} &:= (\theta \otimes 
\left(\bigotimes_{i=1}^{n-1} q\right))_{n\in\mathbb{N}}.\\
\end{align*}

By the above limits \eqref{eqlim1} and \eqref{eqlim3} and by the fact that $L_1$-convergence implies almost-sure convergence of a subsequence, we can choose a Borel set $\Phi \subseteq E$, $\nu_1(\Phi) = 1$ such that for all $x\in \Phi$ and a subsequence (again labeled by $n\in\mathbb{N}$) it holds
\begin{align}
\label{Feqdx}
\lim_{n\rightarrow \infty} \int_{E^n} \left| F\circ L_n - F(\nu_1) \right| \mu^{(\delta_x,p_1,n)} = 0,
\end{align}
and
\begin{align}
&\lim_{n\rightarrow \infty} \int_{E^n} \left| \frac{1}{n} \sum_{i=1}^{n-1} \beta(p_1(x_i),\pi(x_i)) - \beta_2^{\nu_1}(\nu_1 \otimes p_1)\right| \mu^{(\delta_x,p_1,n)}(dx_1,...,dx_n) = 0,\\
\label{betaeqdx}
\Rightarrow &\lim_{n\rightarrow \infty} \beta_n^{\delta_x}(\mu^{(\delta_x,p_1,n)}) = \beta_2^{\nu_1}(\nu_1\otimes p_1).
\end{align}
Since $\nu_1$ and $\mu^*$ are equivalent by Lemma \ref{862b}, $\mu^*(\Phi) = 1$. Since $\mu^*(\Phi) = 1$, it holds $\pi^{(l_0)}(\Phi) = 1$, as otherwise Lemma \ref{862b} would imply $\mu^*(\Phi^C) > 0$. So we found the set $\Phi$ mentioned at the beginning of the proof and the required subsequence. It remains to show \eqref{ineqlb1} for all $x\in \Phi$ and $x_1,...,x_{l_0} \in E$.

Let $x_1,...,x_{l_0} \in E$. By \eqref{Feqdx}, dominated convergence and the triangle inequality, it holds
\begin{align*}
&\int_{E^{n-l_0}} \left| F\circ L_n(x_1,...,x_n) - F(\nu_1) \right| \mu^{(\delta_x,p_1,n-l_0)}(dx_{l_0+1},...,dx_n) \\
&\leq \int_{E^{n-l_0}} \left| F\circ L_n(x_1,...,x_n) - F\circ L_{n-l_0}(x_{l_0+1},...,x_n)\right| \mu^{(\delta_x,p_1,n-l_0)}(dx_{l_0+1},...,dx_n)\\ &+ \int_{E^{n-l_0}} \left| F\circ L_{n-l_0}(x_{l_0+1},...,x_n) - F(\nu_1)\right| \mu^{(\delta_x,p_1,n-l_0)}(dx_{l_0+1},...,dx_n) \\&\rightarrow 0,
\end{align*}
since $F$ is continuous and $\|L_n(x_1,...,x_{l_0},\cdot) - L_{n-l_0}\|_v \leq 2 l_0/n \rightarrow 0$, where $\|\cdot\|_v$ denotes the total variation norm.
Thus
\begin{align}
\int_{E^{n-l_0}} F\circ L_n(x_1,...,x_n) \mu^{(\delta_x,p_1,n-l_0)}(dx_{l_0+1},...,dx_n) \rightarrow F(\nu_1).
\end{align}
Finally, it follows
\begin{align*}
&\liminf_{n\rightarrow \infty} \frac{1}{n} \rho_{n-l_0}^{\delta_x}(n F \circ L_{n-l_0}(x_1,...x_{l_0},\cdot))\\ &= \liminf_{n\rightarrow \infty} \sup_{\nu \in \mathcal{P}(E^{n-l_0})} \left( \int_{E^{n-l_0}} F \circ L_n(x_1,...,x_n) \nu(dx_{l_0+1},...,dx_n) - \beta_{n-l_0}^{\delta_x}(\nu) \right)\\
&\geq \liminf_{n\rightarrow \infty} \left(\int_{E^{n-l_0}} F \circ L_n(x_1,...,x_n) \mu^{(\delta_x,p_1,n-l_0)}(dx_{l_0+1},...,dx_n) - \beta_{n-l_0}^{\delta_x}(\mu^{(\delta_x,p_1,n-l_0)})\right)\\
&= F(\nu_1) - \beta_2^{\nu_1}(\nu_1 \otimes p_1)\\
& \geq \sup_{\nu\in\mathcal{P}(E)}\left( F(\nu) - I(\nu)\right) - 3 \varepsilon.
\end{align*}

\textbf{Second Step:}
First, define $g_n : E^{l_0} \rightarrow \mathbb{R}$ for $n > l_0$ by
\[
g_n(x_1,...,x_{l_0}) = \int_{\Phi} \rho_{n-l_0}^{\delta_x}(n F \circ L_n(x_1,...,x_{l_0},\cdot)) \pi(x_{l_0},dx).
\]
Then $g_n$ is upper semi-analytic, since $(x,x_1,...,x_{l_0}) \mapsto \rho_{n-l_0}^{\delta_x}(n F \circ L_n(x_1,...,x_{l_0},\cdot))$ is (by \cite[Prop. 7.47 and 7.48]{bertsekas2004stochastic}, see also Lemma \ref{ughselection}) and thus $g_n$ is as well (by \cite[Prop. 7.48]{bertsekas2004stochastic}).

By Fatou's Lemma (applicable since $|\frac{1}{n}\rho_n^{\delta_x}(nF\circ L_n)| \geq -\|F\|_{\infty}$), for all $x_1,...,x_{l_0} \in E$, it holds
\[\liminf_{n\rightarrow \infty} \frac{1}{n} g_n(x_1,...,x_n) \geq  \int_{\Phi} \liminf_{n\rightarrow \infty} \frac{1}{n} \rho_{n-l_0}^{\delta_x}(n F \circ L_n(x_1,...,x_{l_0},\cdot)) \pi(x_{l_0},dx) \geq \sup_{\nu\in\mathcal{P}(E)}(F(\nu)-I(\nu))-3\varepsilon.\]
We define the sets
\[
\Omega_n := \left\{(x_1,...,x_{l_0})\in E^{l_0} :\frac{1}{n} g_j(x_1,...,x_{l_0}) \geq \sup_{\nu\in\mathcal{P}(E)}(F(\nu)-I(\nu))-4\varepsilon \text{ for all } j \geq n\right\}
\]
for $n\in \mathbb{N}$, which are universally measurable and satisfy $\Omega_1 \subseteq \Omega_2 \subseteq \Omega_3 ...$ and
$
\cup_{i=1}^{\infty} \Omega_i = E^{l_0}.
$
For $n\in \mathbb{N}$ let $p_n := \mu^{\pi_0,\pi,l_0}(\Omega_n)$. Then by continuity from below it holds $p_n \rightarrow 1$ for $n\rightarrow \infty$.
We have by Lemma \ref{allpi0lower}, Lemma \ref{ughselection} and monotonicity of $\rho_{l_0}$
\begin{align*}
\liminf_{n\rightarrow \infty} \frac{1}{n} \rho_n(nF\circ L_n) &\geq \liminf_{n\rightarrow \infty} \frac{1}{n} \rho_{l_0}(g_n)\\
&\geq \liminf_{n\rightarrow \infty} \frac{1}{n} \rho_{l_0}(\eins_{\Omega_n}  n(\sup_{\nu \in \mathcal{P}(E)} (F(\nu)-I(\nu)) - 4\varepsilon) - \eins_{\Omega_n^C}  n \|F\|_{\infty})\\
&\geq \liminf_{n\rightarrow \infty} \left(p_n  \left(\sup_{\nu \in \mathcal{P}(E)} (F(\nu)-I(\nu)) - 4\varepsilon\right) - (1-p_n)  \|F\|_{\infty}\right)\\
&= \sup_{\nu \in \mathcal{P}(E)} (F(\nu)-I(\nu)) - 4\varepsilon,
\end{align*}
where the last inequality uses $\beta(\nu,\nu) = 0$ for all $\nu \in \mathcal{P}(E)$.
\qed

\subsection{Main Theorem Upper Bound}
\label{subsec::UpperBound}
\subsubsection{Preliminary Results}
\label{subsubsec::UpperBoundPre}
\begin{lemma}
	\label{tupeltight}
	Let $\hat{E}$ be another Polish space, $(X_n)_{n\in\mathbb{N}}$ be an 
	$E$-valued sequence of random variables and $(Y_n)_{n\in\mathbb{N}}$ be a 
	$\hat{E}$-valued sequence of random variables. If both 
	$(X_n)_{n\in\mathbb{N}}$ and $(Y_n)_{n\in\mathbb{N}}$ are tight, then 
	$\left((X_n,Y_n)\right)_{n\in\mathbb{N}}$ is also tight.
	\begin{proof}
		Let $\varepsilon > 0$ and choose $A \subseteq E$ and $B \subseteq \hat{E}$ 
		both compact such that
		\begin{align*}
		\mathbb{P}(X_n \in A) & \geq 1-\varepsilon \\
		\mathbb{P}(Y_n \in B) & \geq 1-\varepsilon 
		\end{align*}
		holds for all $n\in\mathbb{N}$. Then $A\times B$ is compact in $E\times 
		\hat{E}$, and
		\[
		\mathbb{P}((X_n,Y_n)\in A\times B) = \mathbb{P}(X_n \in A) - \mathbb{P}((X_n,Y_n)\in A\times B^C) \geq \mathbb{P}(X_n \in A) - \mathbb{P}(Y_n \in B^C) \geq 1-2\varepsilon.
		\]\end{proof}
\end{lemma}

The following theorem is essential for the proof the upper bound. It is 
based on Proposition 8.2.5 and Theorem 8.2.8 in \cite{dupuis2011weak}.
\begin{theorem}
	\label{Tool1}
	Assume (T) and let $(\mu^{(n)})_{n\in\mathbb{N}} \subseteq \mathcal{P}(E^n)$ be a sequence of measures such that 
	\[
	\sup_{n\in\mathbb{N}} \frac{1}{n}\beta^{\pi_0}_n(\mu^{(n)}) < \infty.
	\]
	For $n\in\mathbb{N}$, let $X_n = (X_{n,1},...,X_{n,n})$ be $E^n$-valued 
	random variables with distribution $\mu^{(n)}$.
	Define the sequence of $\mathcal{P}(E\times E)$-valued random variables 
	$(\gamma_n)_{n\in \mathbb{N}}$ by
	\[
	\gamma_{n-1} := \frac{1}{n-1}\sum_{i=1}^{n-1} \delta_{X_{n,i}} \otimes 
	\mu^{(n)}_{i,i+1}(X_{n,1},...,X_{n,i}).
	\]
	It holds:
	\begin{itemize}
		\item[(i)] $(\gamma_n)_{n\in\mathbb{N}}$ is tight.
		\item[(ii)] For every convergent (in distribution) subsequence of 
		$(\gamma_n)_{n\in\mathbb{N}}$, there exists a probability space 
		$(\bar{\Omega},\bar{\mathcal{F}},\bar{\mathbb{P}})$, 
		such that on this space, there exist random variables 
		$\bar{\gamma}_n \sim \gamma_n$ and $\bar{\gamma}\sim \gamma $ 
		with $\bar{\gamma}_n \stackrel{w}{\rightarrow} \bar{\gamma}$ 
		$\bar{\mathbb{P}}$-a.s.. Further, $\bar{\gamma}^{(1)} = 
		\bar{\gamma}^{(2)}$ $\bar{\mathbb{P}}$-a.s., where $\bar{\gamma}^{(1)}$ and $\bar{\gamma}^{(2)}$ are the first and second marginals of $\bar{\gamma}$.
	\end{itemize}
	\begin{proof}
		For the proof of (i), there is nothing to show if (T.1') holds. So we only consider the case that (T.1) holds. Define the sequence of first marginals 
		$(\tilde{L}_n)_{n\in\mathbb{N}} := (\gamma_{n}^{(1)})_{n\in\mathbb{N}}$.
		We first show that $(\tilde{L}_n)_{n\in\mathbb{N}}$ is tight.
		The idea is to use (T.1) which yields a tightness function $c$ on $E$ 
		defined by 
		\[
		c(x) := U(x) - \rho^{\pi(x)}(U),
		\]
		and thus a tightness function $G$ on $\mathcal{P}(E)$ defined by
		\[
		G(\theta) := \int_E c d \theta,
		\]
		where we refer to Appendix A.3.17 of \cite{dupuis2011weak} and the 
		preceding 
		definition, as well as Lemma 8.2.4 of \cite{dupuis2011weak} for 
		properties of a 
		tightness function.
		In the following, we show that $\mathbb{E}\left[\int_E c d \tilde{L}_n\right] \leq K 
		\in \mathbb{R}$ uniformly in $n \in \mathbb{N}$, which is sufficient to yield the 
		claim since 
		\[
		\mathbb{E}\left[\int_E c d\tilde{L}_n \right] = \int_{\mathcal{P}(E)} 
		\left(\int_E c d\theta\right) \mathbb{P}\circ\tilde{L}_n^{-1}(d\theta)
		\]
		and the set $\{Q\in \mathcal{P}(\mathcal{P}(E)) : \int_{\mathcal{P}(E)} 
		G(\theta) Q(d\theta) \leq M \}$ is tight for every $M\in \mathbb{R}$ by 
		Lemma 8.2.4 of \cite{dupuis2011weak}.
		
		In a first step, we assume that $U$ is bounded. Then
		for all $x\in E$, by definition of $\rho^{\pi(x)}$, it holds
		\begin{align}
		\label{eqhl1}
		\forall \nu \in \mathcal{P}(E): \beta(\nu,\pi(x)) \geq \int_E U 
		d\nu -\rho^{\pi(x)}(U).
		\end{align}
		
		For $i\in \{1,2,...,n-1\}$, $\mu_{i,i+1}^{(n)}(X_{n,1},...,X_{n,i})$ is 
		a regular conditional distribution of $X_{n,i+1}$ given 
		$\sigma(X_{n,1},...,X_{n,i})$ and therefore
		\[
		\mathbb{E}[U(X_{n,i+1}) | X_{n,1},...,X_{n,i}] = \int_E U
		d\mu_{i,i+1}^{(n)}(X_{n,1},...,X_{n,i}).\footnote{See for example 
			\cite{dudley2002real} Theorem 10.2.5., where we used that $U$ is 
			bounded.}
		\]
		We calculate
		\begin{align*}
		\mathbb{E}[U(X_{n,i+1}) - U(X_{n,i})]
		&= \mathbb{E}\left[\mathbb{E}[U(X_{n,i+1}) | 
		X_{n,1},...,X_{n,i}]-U(X_{n,i})\right] \\
		&= \mathbb{E}\left[\int_E U d\mu_{i,i+1}^{(n)}(X_{n,1},...,X_{n,i}) - 
		U(X_{n,i})\right]\\
		&= \mathbb{E}\left[\int_E U d\mu_{i,i+1}^{(n)}(X_{n,1},...,X_{n,i}) - 
		\rho^{\pi(X_{n,i})}\right] + 
		\mathbb{E}\left[\rho^{\pi(X_{n,i})} - U(X_{n,i})\right]\\
		&\stackrel{\eqref{eqhl1}}{\leq} 
		\mathbb{E}\left[\beta(\mu_{i+1,i}^{(n)}(X_{n,1},...,X_{n,i}),\pi(X_{n,i}))\right]
		- \mathbb{E}\left[c(X_{n,i})\right]
		\end{align*}
		Summing the above inequalities over $i\in\{1,2,...,n-1\}$ gives
		\begin{align*}
		\mathbb{E}\left[U(X_{n,n}) - U(X_{n,1})\right] &\leq \sum_{i=1}^{n-1} 
		\left( 
		\mathbb{E}\left[\beta(\mu_{i,i+1}^{(n)}(X_{n,1},...,X_{n,i}),\pi(X_{n,i}))\right]
		- \mathbb{E}\left[c(X_{n,i})\right] \right) \\
		\Rightarrow \sum_{i=1}^{n-1} \mathbb{E}\left[c(X_{n,i})\right] &\leq 
		\mathbb{E}\left[U(X_{n,1})\right] + \sum_{i=1}^{n-1} 
		\mathbb{E}\left[\beta(\mu_{i,i+1}^{(n)}(X_{n,1},...,X_{n,i}),\pi(X_{n,i}))\right],
		\end{align*}
		where $\mathbb{E}[U(X_{n,n})] \geq 0$ is used. Dividing the above 
		inequality by $(n-1)$, one obtains
		\begin{align*}
		\mathbb{E}\left[\int_E c d\tilde{L}_{n-1}\right] &= \frac{1}{n-1} 
		\sum_{i=1}^{n-1} \mathbb{E}[c(X_{n,i})] \\
		&\leq \frac{1}{n-1} \left(\mathbb{E}\left[U(X_{n,1})\right] + 
		\sum_{i=1}^{n-1} 
		\mathbb{E}\left[\beta(\mu_{i,i+1}^{(n)}(X_{n,1},...,X_{n,i}),\pi(X_{n,i}))\right]\right)\\
		&\leq 
		\frac{1}{n-1} \left(\beta(\mu_{0,1}^{(n)},\pi_0) + \rho^{\pi_0}(U) + 
		\sum_{i=1}^{n-1} 
		\mathbb{E}\left[\beta(\mu_{i,i+1}^{(n)}(X_{n,1},...,X_{n,i}),\pi(X_{n,i}))\right]\right)\\
		&= \frac{1}{n-1} \beta^{\pi_0}_n(\mu^{(n)}) + 
		\frac{1}{n-1}\rho^{\pi_0}(U).
		\end{align*}
		The last term of the above inequality chain is uniformly bounded for 
		all $n \geq 2$ by assumption and part (c) of (T.1), and we denote this 
		bound by $K\in\mathbb{R}$. 
		
		Now, let us show the above for unbounded U. Let $U_k := U\wedge k$ (for 
		$k\in \mathbb{N}$) and $c_k(x) := U_k(x) - \rho^{\pi(x)}(U_k)$.
		We have shown
		\[ 
		\mathbb{E}\left[\int_E c_k d\tilde{L}_{n-1}\right] \leq \frac{1}{n-1} 
		\beta_n^{\pi_0}(\mu^{(n)}) + 
		\frac{1}{n-1}\rho^{\pi_0}(U_k) \leq \frac{1}{n-1} 
		\beta^{\pi_0}_n(\mu^{(n)}) + 
		\frac{1}{n-1}\rho^{\pi_0}(U).
		\]
		One quickly verifies that $c_k \geq c\wedge\left(\inf_{\tau \in 
			\mathcal{P}(E)^2} \beta(\tau)\right)$,\footnote{We separately look at the cases $U(x) \leq k$ and $U(x) \geq k$. It 
			holds $c_k(x) \geq \inf_{\tau \in 
				\mathcal{P}(E)^2} \beta(\tau)$, if $U(x) \geq k$, and $c_k(x) \geq 
			c(x)$, if $U(x) \leq k$.} which is bounded below by a constant by lower boundedness of $\beta$ and (T.1). Further for all $x\in E$, it holds $c(x) = \lim_{k\rightarrow 
			\infty} c_k(x)$ by monotone convergence and
		therefore by Fatou's Lemma
		\[
		\mathbb{E}\left[\int_E c d\tilde{L}_{n-1}\right] \leq 
		\liminf_{k\rightarrow \infty} \mathbb{E}\left[\int_E c_k 
		d\tilde{L}_{n-1}\right] \leq \frac{1}{n-1} 
		\beta^{\pi_0}_n(\mu^{(n)}) + 
		\frac{1}{n-1}\rho^{\pi_0}(U) \leq K.
		\]
		This shows $(\tilde{L}_n)_{n\in\mathbb{N}}$ is tight.
		
		Next, we show that the sequence of second marginals of 
		$(\gamma_n)_{n\in\mathbb{N}}$ is tight, i.e.~we prove tightness of the 
		sequence $(\gamma_n^{(2)})_{n\in\mathbb{N}}$ given by $\gamma_{n-1}^{(2)} = 
		\frac{1}{n-1} \sum_{i=1}^{n-1} \mu_{i,i+1}^{(n)}(X_{n,1},...,X_{n,i})$.
		This follows from
		\begin{align*}
		\mathbb{E}\left[\int_E c d\gamma_n^{(2)}\right] &= \frac{1}{n}\sum_{i = 
			1}^n 
		\mathbb{E}\left[\int_E c 
		d\mu_{i,i+1}^{(n+1)}(X_{n+1,1},...,X_{n+1,i})\right]\\
		&\stackrel{(*)}{=} \frac{1}{n}\sum_{i=1}^{n} 
		\mathbb{E}\left[\mathbb{E}\left[c(X_{n+1,i}) | 
		\mathcal{F}_i^{n+1}\right]\right] \\
		&= \frac{1}{n}\sum_{i=1}^n \mathbb{E}\left[ c(X_{n+1,i})\right]\\
		&= \mathbb{E}\left[\int_E c d\tilde{L}_n\right] \leq K,
		\end{align*}
		where the last inequality is uniformly in $n\in\mathbb{N}$ as shown 
		above. Note that while equality $(*)$ requires integrability, we can 
		circumvent this requirement by the same argumentation as above, in that 
		we first assume $U$ to be bounded and use Fatou's Lemma for the 
		transition to the general case.
		
		Tightness of $(\gamma_n)_{n\in\mathbb{N}}$ now follows from tightness 
		of the marginals $(\gamma_n^{(2)})_{n\in\mathbb{N}}$ and 
		$(\tilde{L}_n)_{n\in\mathbb{N}}$, see Lemma \ref{tupeltight}. \\
		
		For part (ii), choose any subsequence still denoted by
		$(\gamma_n)_{n\in\mathbb{N}}$ that converges in distribution, which 
		means there exists a $\mathcal{P}(E\times E)$ valued random variable 
		$\gamma$ such that
		\[
		\mathbb{P}\circ \gamma_n^{-1}  \stackrel{w}{\rightarrow} 
		\mathbb{P}\circ \gamma^{-1}.
		\]
		
		With Skorohod's representation theorem (see e.g.~\cite[Page 102]{ethier1986markov}), we can go over 
		to a 
		probability space 
		$(\bar{\Omega},\bar{\mathcal{F}},\bar{\mathbb{P}})$ such 
		that on this space, there exist random variables $\bar{\gamma}_n 
		\sim \gamma_n$ and $\bar{\gamma}\sim \gamma$ with 
		$\bar{\gamma}_n \stackrel{w}{\rightarrow} \bar{\gamma}$ 
		$\bar{\mathbb{P}}$-a.s..
		
		It only remains to show that $\bar{\gamma}^{(1)} = 
		\bar{\gamma}^{(2)}$ holds $\bar{\mathbb{P}}$-a.s.. Since 
		$\mu_{i,i+1}^{(n)}(X_{n,1},...,X_{n,i})$ is a regular conditional 
		distribution of $X_{n,i+1}$ given $X_{n,1},...,X_{n,i}$, it holds
		\[
		\label{eq2hl}
		\mathbb{E}\left[\left(f(X_{n,i+1}) - \int_E f 
		d\mu_{i,i+1}^{(n)}(X_{n,1},...,X_{n,i})\right) ~\middle|~ 
		X_{n,1},...,X_{n,i} \right] = 0
		\]
		for $f \in C_b(E)$, $n\in\mathbb{N}, i\in\{1,...,n-1\}$. That means the 
		terms 
		inside the expectation form (for fixed $n$) a martingale difference 
		sequence. For ease of notation, we write 
		\begin{align*}
		a_{n,i} &:= f(X_{n,i}), \\
		b_{n,i} &:= \int_E f d\mu_{i-1,i}^{(n)}(X_{n,1},...,X_{n,i-1}).
		\end{align*}
		and get for $n \geq 2$,
		\begin{align*}
		&\bar{\mathbb{E}}\left[\left(\int_E f d \bar{\gamma}_{n-1}^{(1)}
		- \int_E f d \bar{\gamma}_{n-1}^{(2)}\right)^2 \right]\\		
		&=\mathbb{E}\left[\left(\int_E f d \gamma_{n-1}^{(1)} - \int_E f d 
		\gamma_{n-1}^{(2)}\right)^2 \right] \\
		&= \mathbb{E}\left[\left( \frac{1}{n-1} \sum_{i=1}^{n-1} a_{n,i} - 
		b_{n,i+1}\right)^2\right] \\
		&= \frac{1}{(n-1)^2} \mathbb{E}\left[ 
		\left((a_{n,1}-b_{n,n})+\left(\sum_{i=1}^{n-1} 
		a_{n,i}-b_{n,i}\right)\right)^2\right]\\
		&= \frac{1}{(n-1)^2} \mathbb{E}\left[ (a_{n,1}-b_{n,n})^2 + 
		2(a_{n,1}-b_{n,n})\left(\sum_{i=1}^{n-1} a_{n,i}-b_{n,i}\right) + 
		\left(\sum_{i=1}^{n-1} (a_{n,i}-b_{n,i})^2\right) \right] \\
		&\leq \frac{4+8(n-1)+4(n-1)}{(n-1)^2} \|f\|_{\infty}^2,
		\end{align*}
		which converges to $0$ for $n\rightarrow \infty$.
		By the triangle inequality
		\begin{align*}
		\bar{\mathbb{E}}\left[\left(\int_E f d \bar{\gamma}^{(1)} - 
		\int_E f d\bar{\gamma}^{(2)}\right)^2\right] = 0,
		\end{align*}
		which implies $\int_E f d\bar{\gamma}^{(1)} = \int_E f 
		d\bar{\gamma}^{(2)}$ $\bar{\mathbb{P}}$-a.s. for every $f\in 
		C_b(E)$.
		
		By a similar separation argument as in Lemma \ref{weakcon}, we see that $\mathbb{\bar{P}}$-a.s.
		\[ \forall f \in \mathcal{U}_b(E,m): \int_E f d\bar{\gamma}^{(1)} = 
		\int_E f d\bar{\gamma}^{(2)},
		\]
		where $m$ is an equivalent metric on $E$ as given by \cite[Lemma 3.1.4]{stroock1993probability}. 
		Since $\mathcal{U}_b(E,m)$ is measure determining on $(E,d)$ (see Appendix A.2.2 of \cite{dupuis2011weak}), we 
		conclude $\bar{\gamma}^{(1)} = 
		\bar{\gamma}^{(2)}$ $\mathbb{\bar{P}}$-a.s..
	\end{proof}
\end{theorem}
\subsubsection{Proof of Theorem \ref{mainth} Upper Bound}
Let $F: \mathcal{P}(E) \rightarrow \mathbb{R}$ be bounded and upper semi-continuous.
By definition
\[ \frac {1}{n} \rho_n(nF\circ L_n) = \sup_{\mu \in 
	\mathcal{P}(E^n)}\left( \int_{E^n} F\circ L_n d\mu - \frac{1}{n} 
\beta^{\pi_0}_n(\mu)\right).
\]
Using the boundedness of $F$, 
the lower boundedness of $\beta$ and the fact that $\beta(\nu,\nu) = 0$ for all $\nu \in \mathcal{P}(E)$, one verifies that the 
right-hand side in the above equation is bounded below by $-\| F 
\|_{\infty}$ and bounded above by $\| F \|_{\infty} + \inf_{\tau \in 
	\mathcal{P}(E)^2} |\beta(\tau)|$. Thus for each $n\in \mathbb{N}$, we can 
choose $\mu^{(n)} \in \mathcal{P}(E^n)$ such that
\begin{align}
\label{mteq1}
\frac{1}{n} \rho_n(nF\circ L_n) - \frac{1}{n} \leq \int_{E^n} F\circ 
L_n d\mu^{(n)} - \frac{1}{n} \beta^{\pi_0}_n(\mu^{(n)})
\end{align}
and
\[
\sup_{n\in\mathbb{N}} \frac{1}{n}\beta^{\pi_0}_n(\mu^{(n)}) < \infty.
\]
The latter will be used to apply Theorem \ref{Tool1} in a few moments.
First, we use $\beta(\nu,\nu) = 0$ for all $\nu \in \mathcal{P}(E)$ and convexity of $\beta_2^{\cdot}(\cdot)$ to calculate
\begin{align}
\begin{split}
\label{eq::ubeq2}
&\frac{1}{n} \beta^{\pi_0}_n(\mu^{(n)})\\
&=  \frac{1}{n} \beta(\mu_{0,1}^{(n)},\pi_0) + 
\frac{1}{n}\sum_{i=1}^{n-1} \int_{E^n} 
\beta(\mu^{(n)}_{i,i+1}(x_1,...,x_i),\pi(x_i)) \mu^{(n)}(d 
x_1,...,dx_n)\\
&= \frac{1}{n} \beta^{\pi_0}_2(\mu_{0,1}^{(n)}\otimes 
\pi) + \int_{E^n}\frac{1}{n} \sum_{i=1}^{n-1} 
\beta^{\delta_{x_i}}_2(\delta_{x_i}\otimes \mu^{(n)}_{i,i+1}(x_1,...,x_i)) 
\mu^{(n)}(dx_1,...,dx_n) \\
&\geq \int_{E^n} \beta_2^{\frac{1}{n}\left(\pi_0 + 
	\sum_{i=1}^{n-1}\delta_{x_i}\right)}\left(\frac{1}{n}\left(\mu^{(n)}_{0,1}\otimes\pi+\sum_{i=1}^{n-1}
\delta_{x_i} \otimes \mu^{(n)}_{i,i+1}(x_1,...,x_i)\right)\right)
\mu^{(n)}(dx_1,...,dx_n),\end{split}
\end{align}
where $\otimes$ denotes the product measure if both arguments are measures.

For $n\in\mathbb{N}$, let $X_n = (X_{n,1},...,X_{n,n})$ be $E^n$-valued 
random variables with distribution $\mu^{(n)}$.
Define the sequence of $\mathcal{P}(E\times E)$-valued random variables 
$(\gamma_n)_{n\in \mathbb{N}}$ by
\[
\gamma_{n-1} := \frac{1}{n-1}\sum_{i=1}^{n-1} \delta_{X_{n,i}} \otimes 
\mu^{(n)}_{i,i+1}(X_{n,1},...,X_{n,i}).
\]

For any subsequence, Theorem \ref{Tool1} (i) 
yields a  further subsequence (again labeled by $n \in \mathbb{N}$ and fixed
for the rest of the proof of the upper bound) such that 
$(\gamma_n)_{n\in\mathbb{N}}$ converges in distribution. By Theorem 
\ref{Tool1} (ii), there exists a probability space 
$(\bar{\Omega},\bar{\mathcal{F}},\bar{\mathbb{P}})$, 
such that on this space, there exist random variables 
$\bar{\gamma}_n \sim \gamma_n$ and $\bar{\gamma}\sim \gamma $ 
with $\bar{\gamma}_n \stackrel{w}{\rightarrow} \bar{\gamma}$ 
$\bar{\mathbb{P}}$-a.s.. Further, $\bar{\gamma}^{(1)} = 
\bar{\gamma}^{(2)}$ $\bar{\mathbb{P}}$-a.s., where $\bar{\gamma}^{(1)}$ and $\bar{\gamma}^{(2)}$ are the first and second marginals of $\bar{\gamma}$.

Define $(\bar{L}_n)_{n\in\mathbb{N}} := 
(\bar{\gamma}_{n,1})_{n\in\mathbb{N}}$ and $\bar{L} := 
\bar{\gamma}^{(1)}$, and note 
$\bar{L}_n \stackrel{w}{\rightarrow} \bar{L}$ 
$\mathbb{\bar{P}}$-a.s..
With these definitions, \eqref{mteq1} and \eqref{eq::ubeq2}, we get\footnote{In the formula, $\overline{X}_{n,n}$ are (redefined) random variables on $(\bar{\Omega},\bar{\mathcal{F}},\bar{\mathbb{P}})$ such that $(X_{n,n},\gamma_{n-1}) \sim (\bar{X}_{n,n},\bar{\gamma}_{n-1})$ for all $n\in \mathbb{N}$.}
\[
\frac{1}{n} \rho_n(nF\circ L_n) - \frac{1}{n} \leq 
\bar{\mathbb{E}}\left[F\left(\frac{n-1}{n}\bar{L}_{n-1} + \frac{1}{n} \delta_{\bar{X}_{n,n}}\right) - 
\beta^{\frac{\pi_0}{n} + \frac{n-1}{n} 
	\bar{L}_{n-1}}_2\left(\frac{\mu_{0,1}^{(n)}\otimes \pi}{n} + 
\frac{n-1}{n}\bar{\gamma}_{n-1} \right) \right].
\]

For ease of notation, define
\begin{align*}
t_{n,0} &:= \frac{n-1}{n}\bar{L}_{n-1} + \frac{1}{n} \delta_{\bar{X}_{n,n}},\\
t_{n,1} &:= \frac{\pi_0}{n} + \frac{n-1}{n} 
\bar{L}_{n-1}, \\
t_{n,2} &:= \frac{\mu_{0,1}^{(n)}\otimes \pi}{n} + 
\frac{n-1}{n}\bar{\gamma}_{n-1}.
\end{align*}
and note that $t_{n,0} \stackrel{w}{\rightarrow} \bar{L}$, $t_{n,1} \stackrel{w}{\rightarrow} \bar{L}$ and 
$t_{n,2} \stackrel{w}{\rightarrow} \bar{\gamma}$, all $\bar{\mathbb{P}}$-a.s..

Therefore, by upper semi-continuity of $F$ and
$-\beta_2^{\cdot}(\cdot)$, it holds
\begin{align*}
\limsup_{n\rightarrow \infty} \frac{1}{n}\rho_n(n F \circ L_n) 
&\leq \limsup_{n\rightarrow \infty} \bar{\mathbb{E}}\left[F(t_{n,0}) - \beta_2^{t_{n,1}}(t_{n,2})\right] \\
&\leq \bar{\mathbb{E}}\left[F \circ \bar{L} - 
\beta^{\bar{L}}_2(\bar{\gamma})\right]\\
& = \bar{\mathbb{E}}\left[ F\circ \bar{L} - \int_E 
\beta(\bar{\gamma}_{1,2}(x),\pi(x)) 
\bar{L}(dx)\right]\\
&\leq \sup_{\nu \in \mathcal{P}(E)}\left( F(\nu) - \inf_{q: \nu q = 
	\nu}\int_{E} \beta(q(x),\pi(x)) \nu(dx)\right),
\end{align*}
where the last inequality uses the fact that $\bar{\gamma}^{(1)} = \bar{\gamma}^{(2)}$ holds $\bar{\mathbb{P}}$-a.s.. We have shown that every subsequence has a further subsequence such 
that this inequality holds, which implies it also holds for the whole 
sequence.
\qed
\subsection{Proof of Corollary \ref{cor1}}
\label{subsec::Cor}
\textbf{Claim 1:}
	If $\beta_2^{\cdot}(\cdot)$ is lower semi-continuous, then $I$ is lower semi-continuous. If $\beta_2^{\cdot}(\cdot)$ is convex, then $I$ is convex.
	\begin{proof}
		\textbf{Lower Semi-Continuity:}
		
		Let $\nu_n \stackrel{w}{\rightarrow} \nu \in \mathcal{P}(E)$. We have to show
		\[
		\liminf_{n\rightarrow \infty} I(\nu_n) \geq I(\nu).
		\]
		Note that $I$ is bounded below. If the left hand side of the above inequality equals infinity, then there is nothing to prove. So for any subsequence we can choose a further subsequence still denoted by $(\nu_n)_{n\in\mathbb{N}}$ such that $I(\nu_n) < \infty$ for all $n$. Thus, we can choose stochastic kernels $q_n$ such that
		\[
		\beta_2^{\nu_n}(\nu_n \otimes q_n) \leq I(\nu_n) + \frac{1}{n} ~\text{ and }~ \nu_n q_n = \nu_n.
		\]
		Since $\nu_n q_n = \nu_n$ and the sequence $(\nu_n)_{n\in\mathbb{N}}$ is tight by Prokhorov, the sequence $(\nu_n \otimes q_n)_{n\in\mathbb{N}}$ is tight as well (see Lemma \ref{tupeltight}). We go over to a further subsequence still denoted by $(\nu_n\otimes q_n)_{n\in\mathbb{N}}$ such that $\nu_n\otimes q_n \rightarrow \nu \otimes q$, where $\nu q = \nu$ follows by convergence of the marginals. By lower semi-continuity of $\beta_2^{\cdot}(\cdot)$
		\[
		\liminf_{n\rightarrow \infty} I(\nu_n) \geq \liminf_{n\rightarrow \infty} \beta_2^{\nu_n}(\nu_n \otimes q_n) - \frac{1}{n} \geq \beta_2^{\nu}(\nu \otimes q) \geq I(\nu).
		\]
		
		\textbf{Convexity:}
		
		Note $I(\nu) = \inf_{\substack{\tau \in \mathcal{P}(E^2):\\\tau_1 = \tau_2 = \nu}} \beta_2^{\nu}(\tau)$. Let $\nu_1,\nu_2 \in \mathcal{P}(E)$ and $\tau^{(1)},\tau^{(2)} \in \mathcal{P}(E^2)$ with $\tau^{(1)}_1 = \tau^{(1)}_2 = \nu_1, \tau^{(2)}_1 = \tau^{(2)}_2 = \nu_2$. Then
		\begin{align*}
		\lambda \beta_2^{\nu_1}(\tau^{(1)}) + (1-\lambda) \beta_2^{\nu_2}(\tau^{(2)})
		&\geq \beta_2^{\lambda \nu_1 + (1-\lambda) \nu_2}(\lambda \tau^{(1)} + (1-\lambda) \tau^{(2)})\\
		&\geq \inf_{\substack{\tau \in \mathcal{P}(E^2):\\\tau_1 = \tau_2 = \lambda \nu_1 + (1-\lambda) \nu_2}} \beta_2^{\lambda \nu_1 + (1-\lambda) \nu_2}(\tau) = I(\lambda \nu_1 + (1-\lambda) \nu_2).
		\end{align*}
		Taking the infimum on the left hand side over all such $\tau^{(1)}$ and $\tau^{(2)}$ yields the claim.
	\end{proof}
\noindent
\textbf{Claim 2:}
	If the main Theorem \ref{mainth} upper bound holds, and additionally $I$ has compact sub-level sets,
	then the main theorem upper bound extends to all functions $F: \mathcal{P}(E) \rightarrow [-\infty,\infty)$ which are upper semi-continuous and bounded from above.
	\begin{proof}
		Let $F: \mathcal{P}(E) \rightarrow [-\infty,\infty)$ be upper semi-continuous and bounded from above. Define $F_m := -m \vee F$ ($m\in \mathbb{N}$). By assumption, for all $m\in \mathbb{N}$,
		\[
		\liminf_{n\rightarrow \infty} \frac{1}{n} \rho_n(nF \circ L_n) \leq \liminf_{n\rightarrow \infty} \frac{1}{n} \rho_n(nF_m \circ L_n) \leq \sup_{\nu \in \mathcal{P}(E)} \left( F_m(\nu) - I(\nu)\right),
		\]
		so it only remains to show that 
		\[
		\limsup_{m\rightarrow \infty} S_m := \limsup_{m\rightarrow \infty} \sup_{\nu \in \mathcal{P}(E)} (F_m(\nu)-I(\nu)) \leq \sup_{\nu \in \mathcal{P}(E)} (F(\nu)-I(\nu)) =: S.
		\]
		$S_m$ are decreasing (for increasing $m$). If $S_m \rightarrow -\infty$, there is nothing to show. So assume $S_m$ are bounded below by $C \in \mathbb{R}$. Choose $\nu_m \in \mathcal{P}(E)$ such that
		\[
		F_m(\nu_m)  - I(\nu_m) \geq S_m - \frac{1}{m} \geq C-1.
		\]
		So $I(\nu_m)$ are uniformly bounded. By compact sub-level sets of $I$, for any subsequence we can choose a further subsequence still denoted by $(\nu_m)_{m\in\mathbb{N}}$ such that $\nu_m \stackrel{w}{\rightarrow} \nu_\infty$ for some $\nu_\infty \in \mathcal{P}(E)$. Then by upper semi-continuity of $F$ and $- I$,
		\[
		\limsup_{m\rightarrow \infty} F_m(\nu_m) - I(\nu_m) \leq F(\nu_\infty) - I(\nu_\infty) \leq S.
		\]
	\end{proof}

\section{Applications to Robust Markov chains}
\label{sec::Appl}
\subsection{Robust Large Deviations}
In this section $(E,d)$ is assumed to be compact. The main goal of this section is to show Theorem \ref{MainApplicationSum} and illustrate it in Example \ref{simpleexample}. To this end, we show the respective upper bound in Theorem \ref{robustupperbound} and the respective lower bound in Lemma \ref{ldplowerbound}. The intermediate results in this section are concerned with representation formulas for the functionals $\beta_n$ (see Lemma \ref{robustentropyselection} and \ref{robustentropyselection2}) and the verification of conditions (B.1) and (B.2) (see Lemma \ref{Mnconv}, \ref{M2closed} and \ref{lscLDP}).

In the following part leading up the Theorem \ref{robustupperbound}, we assume that $\pi$ satisfies the Feller property. We work with
\[
\beta(\nu,\mu) := \inf_{\hat{\mu} : d_W(\mu,\hat{\mu}) \leq r} 
R(\nu,\hat{\mu}) = \inf_{\hat{\mu} \in M_1(\mu)} R(\nu,\hat{\mu})
\]
for some $r \geq 0$ fixed.  Recall
\begin{align*}M_n(\theta) :=& \{\nu \in \mathcal{P}(E^n) : d_W(\nu_{0,1},\theta)\leq r \text{ and } d_W(\nu_{i,i+1}(x_1,...,x_i),\pi(x_i)) \leq r
~\nu\text{-a.s. for } 
i=1,...,n-1\}.
\end{align*}
To be precise, the above definition requires the condition
$d_W(\nu_{i,i+1}(x_1,...,x_i),\pi(x_i)) \leq r$ to hold for $\nu$-almost all 
$(x_1,...,x_n)\in E^n$ for every decomposition of $\nu$, where the respective 
$\nu$-null set may depend on the given decomposition.
Equivalently, the definition could state that there has to exist one decomposition of $\nu$ such that this condition holds point-wise.
That this notion is equivalent follows by the fact that decompositions of $\nu$ are only unique up to $\nu$-almost-sure equality.
\begin{lemma}{(See also \cite[Lemma 4.4]{bartl2016exponential} and \cite[Prop.~5.2]{lacker2016non})}
	\label{robustentropyselection}
	For all $n\in \mathbb{N}$, it holds
	\[
	\beta_n^{\theta}(\nu) = \inf_{\hat{\mu} \in M_n(\theta)} R(\nu,\hat{\mu}).
	\]
	\begin{proof}
		Fix $\theta \in \mathcal{P}(E)$.
		Define the sets \[Q_0 := \{\hat{\mu} \in \mathcal{P}(E) : d_W(\hat{\mu},\theta) \leq r\}\]
		and for $i = 1,...,n-1$ and $x_1,...,x_i \in E$
		\[
		Q_i(x_1,...,x_i) := \{\hat{\mu} \in \mathcal{P}(E) : d_W(\hat{\mu},\pi(x_i)) \leq r\}.
		\]
		We note that $M_n(\theta) = Q_0 \otimes Q_1 \otimes ... \otimes Q_{n-1}$, where $Q_0 \otimes Q_1 \otimes ... \otimes Q_{n-1}$ is defined as the set of measures $\mu = K_0 \otimes K_1 \otimes K_2 \otimes ... \otimes K_{n-1} \in \mathcal{P}(E^n)$, where $\mu_0 \in K_0$ and $K_i : E^{i} \rightarrow \mathcal{P}(E)$ are Borel measurable kernels such that $K_i(x_1,...,x_i) \in Q_i(x_1,...,x_i)$ for $\mu$-almost all $x_1,...,x_i$. 
		Since for all $i = 1,...,n$, the set $\{(x_1,...,x_i,\hat{\mu}) \in E^i \times \mathcal{P}(E) : \hat{\mu} \in Q_i(x_1,...,x_i)\}$ is trivially Borel, a measurable selection argument (e.g.~\cite[Prop.~7.50]{bertsekas2004stochastic}) yields for $\nu \in \mathcal{P}(E^n)$
		\begin{align*}
		\inf_{\hat{\mu} \in M_n(\theta)} R(\nu,\hat{\mu}) &= \inf_{K_0\otimes ... \otimes K_{n-1} \in Q_0 \otimes ... \otimes Q_{n-1}} \sum_{i=0}^{n-1} \int_{E^n} R(\nu_{i,i+1}(x_1,...,x_i),K_i(x_1,...,x_i)) \nu(dx_1,...,dx_n) \\
		&\stackrel{(*)}{=} \sum_{i=0}^{n-1} \int_{E^n} \inf_{\hat{\mu} \in Q_i(x_1,...,x_i)} R(\nu_{i,i+1}(x_1,...,x_i),\hat{\mu}) \nu(dx_1,...,dx_n)\\
		&= \beta(\nu_{0,1},\theta) + \sum_{i=1}^{n-1} \int_{E^n} \beta(\nu_{i,i+1}(x_1,...,x_i),\pi(x_i)) \nu(dx_1,...,dx_n)\\
		&= \beta_n^{\theta}(\nu),
		\end{align*}
		where rigorously step $(*)$ works inductively, see the proofs of \cite[Lemma 4.4]{bartl2016exponential} and \cite[Prop.~5.2]{lacker2016non}.
	\end{proof}
\end{lemma}
\begin{lemma}
	\label{Mnconv}
	Let $\theta_1,\theta_2 \in \mathcal{P}(E),~\nu_1 \in M_2(\theta_1),~\nu_2 \in M_2(\theta_2)$ and $\lambda \in (0,1)$. Then
	\[\lambda \nu_1 + (1-\lambda) \nu_2 \in M_2(\lambda \theta_1 + (1-\lambda) \theta_2).
	\]
	\begin{proof}
		Write $\nu_1 = \mu_1 \otimes K_1$ and $\nu_2 = \mu_2 \otimes K_2$ for some $\mu_1,\mu_2 \in \mathcal{P}(E)$ and $K_1,K_2$ stochastic kernels on $E$. Further, $K_1$ and $K_2$ are chosen such that $d_W(K_i(x),\pi(x)) \leq r$ for all $x \in E$ and $i\in \{1,2\}$. We have the equality
		\begin{align}
		\label{eqa1}
		\lambda \nu_1 + (1-\lambda)\nu_2 = (\lambda \mu_1 + (1-\lambda)\mu_2) \otimes K,
		\end{align}
		where $K: E \rightarrow \mathcal{P}(E)$ is defined by
		\begin{align*}
		K(x) &= \frac{d\mu_1}{d(\lambda\mu_1 + (1-\lambda)\mu_2)}(x) \lambda K_1(x) + \frac{d\mu_2}{d(\lambda\mu_1 + (1-\lambda)\mu_2)}(x) (1-\lambda) K_2(x)\\
		&=: \lambda_x K_1(x) + (1-\lambda_x) K_2(x).
		\end{align*}
		Equation \ref{eqa1} obviously holds for Borel sets of the form $A \times B \subseteq E^2$, which extends the equality to arbitrary Borel sets by Carath\'eodory. So $K$ is a point-wise convex combination of $K_1$ and $K_2$.
		Since for the first Wasserstein distance the Kantorovich duality (see e.g.~\cite[Chapter 5]{villani2008optimal}) implies
		\begin{align*}
		d_W(\lambda \mu_1 + (1-\lambda) \mu_2, \lambda \theta_1 + (1-\lambda) \theta_2) \leq \lambda d_W(\mu_1,\theta_1) + (1-\lambda) d_W(\mu_2,\theta_2) \leq r
		\end{align*}
		and for all $x\in E$
		\begin{align*} 
		d_W(\lambda_x K_1(x) + (1-\lambda_x) K_2(x), \lambda_x \pi(x) + (1-\lambda_x) \pi(x)) \\\leq \lambda_x d_W(K_1(x),\pi(x)) + (1-\lambda_x) d_W(K_2(x),\pi(x)) \leq r,
		\end{align*}
		the claim follows.
	\end{proof}
\end{lemma}
That $\beta_2^{\cdot}(\cdot)$ is convex follows by the previous lemma and convexity of $R(\cdot,\cdot)$, since
\begin{align*}
	\beta_2^{\lambda \theta_1 + (1-\lambda) \theta_2}(\lambda \nu_1 + (1-\lambda) \ \nu_2) &= \inf_{\hat{\mu} \in M_2(\lambda \theta_1 + (1-\lambda) \theta_2)} R(\lambda \nu_1 + (1-\lambda) \nu_2,\hat{\mu})\\
	&\stackrel{\ref{Mnconv}}{\leq} \inf_{\hat{\mu}_1 \in M_2(\theta_1), \hat{\mu}_2 \in M_2(\theta_2)} R(\lambda \nu_1 + (1-\lambda) \nu_2, \lambda \hat{\mu}_1 + (1-\lambda) \hat{\mu}_2)\\
	&\leq \inf_{\hat{\mu}_1 \in M_2(\theta_1), \hat{\mu}_2 \in M_2(\theta_2)} \lambda R(\nu_1,\hat{\mu}_1) + (1-\lambda) R(\nu_2,\hat{\mu}_2).\\
	&= \lambda \beta_2^{\theta_1}(\nu_1) + (1-\lambda) \beta_2^{\theta_2}(\nu_2)
\end{align*}
It remains to show that $\beta_2^{\cdot}(\cdot)$ is lower semi-continuous. To this end, we first show the following
\begin{lemma}
	\label{M2closed}
	 If $\pi$ satisfies the Feller property, then $M_2(\theta)$ is closed.
	\begin{proof}\footnote{Thanks to Daniel Lacker for providing this proof.}
		Recall $\mu \otimes K \in M_2(\theta)$ if and only if both
		\begin{align}
			d_W(\mu,\theta) &\leq r \label{firstcon},\\
			d_W(K(x),\pi(x)) &\leq r \text{ for } \mu \text{-a.a. } x\in E. \label{secondcon}
		\end{align}
		Condition \eqref{firstcon} is closed (obvious once it is rewritten by Kantorovich duality), so we focus on condition \eqref{secondcon}. Since by assumption $(E,d)$ is compact and thus totally bounded, the set of Lipschitz-1 functions mapping $E$ into $\mathbb{R}$ which are absolutely bounded by 1 (denoted by $\Lip$)  is separable with respect to the sup-norm (follows since the space of uniformly bounded and continuous functions is separable and every subset of a separable metric space is again separable). We denote by $\{f_1,f_2,...\} \subseteq \Lip$ a countable dense subset. Further we are going to use the fact that for bounded and measurable functions $h:E \rightarrow \mathbb{R}$ and $\nu \in \mathcal{P}(E)$ it holds 
		\begin{align*}
			\left(h \geq 0~ \nu-\text{a.s.} \right) \Leftrightarrow \left(\forall g\in C_b(E),g\geq 0: \int_E g(x) h(x) \nu(dx) \geq 0\right),
		\end{align*}
		which is true because $E$ is a Polish space and thus the function $\eins_A$ for the Borel set $A := \{h < 0\}$ can be approximated in $L_1(\nu)$ by a sequence of non-negative, continuous and bounded functions.
		
		We can rewrite condition \eqref{secondcon} as follows
		\begin{align*}
			&d_W(K(x),\pi(x)) \leq r \text{ for } \mu \text{-a.a.~} x\in E \\
			\Leftrightarrow &\left( \forall f \in \Lip: \int_Ef dK(x) - \int_E f d \pi(x) \leq r\right) \text{ for } \mu \text{-a.a.~} x\in E \\
			\Leftrightarrow &\left( \forall i \in \mathbb{N}: \int_Ef_i dK(x) - \int_E f_i d \pi(x) \leq r\right) \text{ for } \mu \text{-a.a.~} x\in E \\
			\Leftrightarrow &\left( \forall i \in \mathbb{N}, \forall g \in C_b(E), g\geq 0 : \int_E g(x) \left( \int_E f_i(y) K(x,dy) - \int_E f_i(y) \pi(x,dy) - r \right) \nu(dx) \leq 0 \right)\\
			\Leftrightarrow &\left( \forall i \in \mathbb{N}, \forall g \in C_b(E), g\geq 0 : \int_{E^2} g(x) f_i(y) \nu \otimes K (dx,dy) - \int_E g(x) \left(\int_E f_i d\pi(x) - r \right) \nu(dx)  \leq 0 \right)
		\end{align*}
		and the last line expresses a closed condition if $\pi$ satisfies the Feller property, which guarantees that $x \mapsto \int_E f d\pi(x)$ is continuous for all $f \in C_b(E)$.
	\end{proof}
\end{lemma}
\begin{lemma}
	\label{lscLDP}
	$\beta_2^{\cdot}(\cdot)$ is lower semi-continuous.
	\begin{proof}
		Let $(\theta_n,\nu_n) \stackrel{w}{\rightarrow} (\theta,\nu) \in \mathcal{P}(E) \times \mathcal{P}(E^2)$ as $n\rightarrow \infty$. We have to show that \[\liminf_{n\rightarrow \infty} \beta_2^{\theta_n}(\nu_n) \geq \beta_2^{\theta}(\nu),\]
		which is done by choosing an arbitrary subsequence and showing there exists a further subsequence such that the inequality holds. So we start with a subsequence still denoted by $(\theta_n,\nu_n)_{n\in\mathbb{N}}$. Let $\hat{\mu}_n \in M_2(\theta_n)$ such that \[\beta_2^{\theta_n}(\nu_n) \geq R(\nu_n,\hat{\mu}_n) - \frac{1}{n}\] and choose a further subsequence still denoted by $(\theta_n,\nu_n)_{n\in\mathbb{N}}$ such that $d_W(\theta_n,\theta) \leq \frac{1}{n}$ and $\hat{\mu}_n$ converges weakly to some $\hat{\mu} \in \mathcal{P}(E^2)$. We show that $\hat{\mu} \in M_2(\theta)$. To this end, define \[M_2^{r,n}(\theta) := \left\{ \mu \otimes K \in \mathcal{P}(E^2) : d_W(\mu,\theta) \leq r+\frac{1}{n},~ d_W(K(x),\pi(x))\leq r \text{ for } \mu\text{-a.a.~} x\in E\right\},\]
		which is closed, as the proof of the previous lemma trivially carries over to this set.
		We see that $\hat{\mu}_m \in M_2^{r,n}(\theta)$ for all $m\geq n$, and therefore $\hat{\mu} \in M_2^{r,n}(\theta)$ for all $n\in \mathbb{N}$, which yields $\hat{\mu} \in M_2(\theta)$. Finally, we get by lower semi-continuity of $R(\cdot,\cdot)$
		\[
		\liminf_{n\rightarrow \infty} \beta_2^{\theta_n}(\nu_n) \geq \liminf_{n\rightarrow \infty} R(\nu_n,\hat{\mu}_n) \geq R(\nu,\hat{\mu}) \geq \inf_{\mu \in M_2(\theta)} R(\nu,\mu) = \beta_2^{\theta}(\nu).
		\]
	\end{proof}
\end{lemma}
The rate function $I$ corresponding to the choice of $\beta$ as defined at the beginning of the section is given by
\[
I(\nu) := \inf_{q:\\ \nu q = \nu} \int_E \inf_{K_x \in M(\pi(x))} R(q(x),K_x) \nu(dx)
\]
for $\nu \in \mathcal{P}(E)$. Using the above observations to apply the main theorem, we get the following:
\begin{theorem}
	\label{robustupperbound}
	For all functions $F : \mathcal{P}(E) \rightarrow [-\infty,\infty)$ which are upper semi-continuous and bounded from above it holds
	\[
	\limsup_{n\rightarrow \infty} \sup_{\mu \in M_n(\pi_0)} \frac{1}{n} \ln \int_{E^n} 
	\exp(n F \circ L_n) d\mu \leq \sup_{\nu \in \mathcal{P}(E)} \left( F(\nu) - 
	I(\nu) \right).
	\]	
	Further, for all closed sets $A \subseteq \mathcal{P}(E)$ it holds
	\[
	\limsup_{n\rightarrow \infty} \sup_{\mu \in M_n(\pi_0)} \frac{1}{n} \ln \mu(L_n \in A) \leq -\inf_{\nu \in A} I(\nu).
	\]
	\begin{proof}
		For the first claim, apply Theorem \ref{mainth}, which
		by compactness of $E$ and thus by Corollary \ref{cor1} extends to all functions $F : \mathcal{P}(E) \rightarrow [-\infty,\infty)$ which are upper semi-continuous and bounded from above. Specifically, for a closed set $A \subseteq \mathcal{P}(E)$ and $F = -\infty \eins_{A^C}$ the second claim follows.
	\end{proof}
\end{theorem}
For the large deviations bound in Theorem \ref{robustupperbound} to be non-vacuous for a closed set $A \subseteq \mathcal{P}(E)$ requires
\begin{align}
	\label{sense}
	\inf_{\nu \in A} I(\nu) > 0.
\end{align}

Intuitively, $\eqref{sense}$ holds if and only if for all pairs $\nu \in A$ and $q$ with $\nu q = \nu$, there is some Borel set $S \subseteq E$ with $\nu(S) > 0$ such that $d_W(q(x),\pi(x)) > r$ for all $x\in S$.

To properly address the question whether the attained bound is sharp, one needs a lower bound in accordance with the upper bound.
The choice of $\beta$ that leads to Theorem \ref{robustupperbound} cannot yield a lower bound with our approach, since condition (B.3) is not satisfied for $r > 0$ and hence the lower bound of Theorem \ref{mainth} cannot be applied.

In the following we therefore consider the functional $\underline{\beta}$ which is chosen such that it resembles $\beta$ and satisfies (B.3), albeit at the cost of not satisfying (B.2). This will lead to the lower bound of Theorem \ref{MainApplicationSum} proven in Lemma \ref{ldplowerbound}. Define
\begin{align*}
\underline{\beta}(\nu,\mu) &:= \inf_{\substack{\hat{\mu} : d_W(\mu,\hat{\mu}) \leq r,\\ \hat{\mu} \ll \mu}} 
R(\nu,\hat{\mu}),\\
\underline{M}_n(\theta) &:= \{\nu \in M_n(\theta) : \nu \ll \theta \otimes \pi \otimes ... \otimes \pi\},\\
\underline{I}(\nu) &:= \inf_{q: \nu q = \nu} \int_E\inf_{K_x \in \underline{M}(\pi(x))} R(q(x),K(x)) \nu(dx)
\end{align*}
Further, we assume for the analysis of the lower bound that $\pi$ satisfies (M), but longer has to satisfy the Feller property.
We find
\begin{lemma}{(See also \cite[Lemma 4.4]{bartl2016exponential} and \cite[Prop.~5.2]{lacker2016non})}
	\label{robustentropyselection2}
	For all $n\in\mathbb{N}$, it holds
	\[
	\underline{\beta}_n^{\theta}(\nu) = \inf_{\mu \in \underline{M}_n(\theta)} R(\nu,\mu).
	\]
	\begin{proof}
		The proof is the same as that of Lemma \ref{robustentropyselection}, except here we need measurability of the sets
		\[
		S_i := \{ (x_1,...,x_i,\hat{\mu}) \in E^i \times \mathcal{P}(E) : d_W(\hat{\mu},\pi(x_i)) \leq r \text{ and } \hat{\mu} \ll \pi(x_i)\}
		\]
		for $i\in\{1,...,n-1\}$. That these sets are indeed Borel measurable can be seen as follows: Define the function $g : \mathcal{P}(E) \times \mathcal{P}(E) \times E \rightarrow \mathbb{R}_+$ by
		\[
		g(\mu,\nu,x) = \frac{d\mu_{|\nu}}{d\nu}(x).
		\]
		Here $\mu_{|\nu}$ denotes the absolutely continuous part of $\mu$ with respect to $\nu$ as given by Lebesgue's decomposition theorem. Then $g$ is Borel as shown in \cite[V.58 and subsequent remark]{dellacherie1982probability}. We have $\mu \ll \nu \Leftrightarrow \int_E g(\mu,\nu,\cdot) d\nu = 1$, which shows that $S_i$ is Borel (as the other conditions that define $S_i$ are trivially Borel).
		
		To arrive at the given form of $\underline{M}_n(\theta)$ one uses the following equivalence for measures $\nu_1,\nu_2\in\mathcal{P}(E)$ and stochastic kernels $K_1,K_2: E \rightarrow \mathcal{P}(E)$ (see e.g.~\cite[Lemma  A.2]{bartl2016exponential})
		\[\left(\nu_1 \otimes K_1 \ll \nu_2 \otimes K_2 \in \mathcal{P}(E^2)\right) \Leftrightarrow \left(\nu_1 \ll \nu_2\text{ and }K_1(x) \ll K_2(x)\text{ for }\nu_1\text{-almost all }x\in E\right).\]
	\end{proof}
\end{lemma}

In complete analogy to the choice of $\beta$ leading to Theorem 
\ref{robustupperbound}, we see that $\underline{\beta}$ satisfies (B.1), which 
is a consequence of the above Lemma \ref{robustentropyselection2} in 
combination with Lemma \ref{Mnconv}, where one additionally uses \[(\mu_1 \ll 
\theta_1 \text{ and } \mu_2 \ll \theta_2) \Rightarrow \lambda \mu_1 + 
(1-\lambda) \mu_2 \ll \lambda \theta_1 + (1-\lambda) \theta_2 \text{ for } 
\mu_1,\mu_2,\theta_1,\theta_2 \in \mathcal{P}(E),~\lambda \in (0,1).\]

As (B.3) and (M) are satisfied as well, Theorem \ref{mainth} yields for all $F\in C_b(\mathcal{P}(E))$
\[
\liminf_{n\rightarrow \infty} \sup_{\mu \in \underline{M}_n(\pi_0)} \frac{1}{n} \int_{E^n} \exp(F \circ L_n) d\mu \geq \sup_{\nu \in \mathcal{P}(E)} \left( F(\nu) - \underline{I}(\nu)\right),
\]
which leads to the following Lemma:
\begin{lemma}
	\label{ldplowerbound}
	Let (M) be satisfied.
	For $G \subseteq \mathcal{P}(E)$ open it holds
	\[ \liminf_{n\rightarrow \infty} \sup_{\mu \in \underline{M}_n(\pi_0)} \frac{1}{n} \ln \mu(L_n \in G) \geq - \inf_{\nu \in G} \underline{I}(\nu).
	\]
	\begin{proof}
		The proof is an adapted version of \cite[Theorem 1.2.3.]{dupuis2011weak}.
		
		We work with the Laplace principle lower bound stated just before the Lemma.
		
		Without loss of generality, assume $\inf_{\nu \in G} \underline{I}(\nu) < \infty$. Let $\nu \in G$ such that $\underline{I}(\nu) < \infty$. Choose $M \in \mathbb{R}$ such that $\underline{I}(\nu) < M$ and $k\in \mathbb{N}$ such that $B(\nu,\frac{1}{k}) := \{ \mu \in \mathcal{P}(E) : \hat{d}(\mu,\nu) \leq \frac{1}{k} \} \subseteq G$, where $\hat{d}$ is some metric on $\mathcal{P}(E)$ compatible with weak convergence. Define
		\[
		h(\theta) := -M\left((\hat{d}(\nu,\theta) \cdot k) \wedge 1\right).
		\]
		We find $-M \leq h \leq 0$, $h(\nu) = 0$ and $h(\theta) = -M$ for $\theta \in B(\nu,\frac{1}{k})^C$.
		Thus for any $\mu \in \mathcal{P}(E^n)$
		\[
		\int_{E^n} \exp(n h \circ L_n) d\mu \leq \exp(-nM) + \mu(L_n \in B(\nu,\delta)) \leq \max\{2\exp(-nM),2\mu(L_n \in B(\nu,\delta))\}.
		\]
		And therefore
		\begin{align*}
		\max\{ \liminf_{n\rightarrow \infty} \sup_{\mu \in \underline{M}_n(\pi_0)} \frac{1}{n} \ln \mu(L_n \in B(\nu,\delta)) , -M \} &\geq \liminf_{n\rightarrow \infty} \sup_{\mu \in \underline{M}_n(\pi_0)} \frac{1}{n} \ln \int_{E^n} \exp(n h\circ L_n) d\mu \\
		&\geq \sup_{\hat{\nu}\in\mathcal{P}(E)} ( h(\hat{\nu}) - \underline{I}(\hat{\nu})) \\
		&\geq h(\nu)-\underline{I}(\nu) = -\underline{I}(\nu).
		\end{align*}
		Since $M > I(\nu)$
		\[
		\liminf_{n\rightarrow \infty} \sup_{\mu \in \underline{M}_n(\pi_0)} \frac{1}{n} \ln \mu(L_n \in B(\nu,\delta)) \geq -\underline{I}(\nu),
		\]
		and using $B(\nu,\delta) \subseteq G$ and the fact that the above reasoning works for all $\nu \in G$ with $\underline{I}(\nu) < \infty$, we get the claim.
	\end{proof}
\end{lemma}
The proof of Theorem \ref{MainApplicationSum} is now done, as it follows from Theorem \ref{robustupperbound} and Lemma \ref{ldplowerbound}.

The following illustrates the obtained results. Note that to calculate the rates, as is usual in large deviations theory, the necessary minimization can be solved efficiently (at least in theory) over convex sets $A$, since $I$ is convex.
\begin{example}
	\label{simpleexample}
	Consider the state space $\{1,2,3\}$ with discrete metric, i.e.~$d(i,j) = 0$ if $i = j$ and $d(i,j) = 1$, else.
	The Markov chain is given by its initial distribution $\pi_0 = \delta_3$ and transition kernel $\pi$ with matrix representation
	\[\begin{bmatrix}
	0.6 & 0.2 & 0.2 \\
	0.3 & 0.4 & 0.3 \\
	0 & 0.3 & 0.7
	\end{bmatrix}.
	\]

	Suppose we are interested in the tail event that the empirical measure $L_n$ under the Markov chain is close (in a certain sense) to the initial distribution $\pi_0$. We are uncertain of the precise model specification of the Markov chain and want to find the worst case (i.e.~slowest possible) convergence rate to zero of this tail event.
	
	Formally, let $r = 0.05$ and take, for $\kappa = 0.2$, the set of measures $A = B_{d_W}(\delta_3,\kappa)$, i.e.~the Wasserstein-1-ball around $\delta_3$ with radius $\kappa$. The set $\{ L_n \in A\}$ models the above mentioned tail event. What is the (exponential) asymptotic rate of convergence of 
	\begin{align}
	\label{eqrate}
	\sup_{\mu \in M_n(\delta_3)} \mu(L_n \in A) \rightarrow 0
	\end{align}
	as $n\rightarrow \infty$? Note that $r$ and the transition kernel are as always implicitly included in $M_n(\delta_3)$.
	
	Calculating the upper bound of Theorem $\ref{MainApplicationSum}$ yields a worst case exponential rate
	\[
	r_{\text{worst case}} \approx 0.0511.
	\]
	This is significantly lower than the normal rate for the Markov chain without the robustness (i.e.~the case $r=0$), which is
	\[
	r_{\text{normal}} \approx 0.0910.
	\]
	Figure 1 showcases the difference in convergence speed.
	Notably, the optimizer of the optimization problem to obtain the worst case rate also yields a kernel $\hat{\pi}$ such that $\pi_0 \otimes \hat{\pi} \otimes ... \otimes \hat{\pi} \in M_n(\pi_0)$ and the Markov chain with transition kernel $\hat{\pi}$ attains the worst case rate, i.e.
	\[
	\pi_0 \otimes \hat{\pi} \otimes ... \otimes \hat{\pi} (L_n \in A) \sim \exp(-n \cdot r_{\text{worst case}}).
	\]
	In other words, the worst case rate in \eqref{eqrate} is obtained and one sequence of optimal measures is Markovian with transition kernel $\hat{\pi}$ given by the matrix
	\[\begin{bmatrix}
	0.6-r & 0.2 & 0.2+r \\
	0.3-r & 0.4 & 0.3+r \\
	0 & 0.3-r & 0.7+r
	\end{bmatrix}.
	\]
	Figure 1 shows a simulated convergence rate for both the initial Markov chain and the Markov chain with worst case transition kernel $\hat{\pi}$ (100 paths simulated) and a comparison of the respective stationary distributions.
	
	Note that in the above example the rates are asymptotically sharp, as the worst-case kernel $\hat{\pi}$ for the rate function is already absolutely continuous with respect to $\pi$, so using $\underline{I}$ instead of $I$ yields the same rate.
	
	Using the above example, one can get an idea when upper and lower bounds of Theorem \ref{MainApplicationSum} may not coincide. If we do not restrict ourselves to $\underline{I}$, it may happen that no optimal kernel $\hat{\pi}$ is absolutely continuous with respect to the initial kernel $\pi$. In that case, we can no longer guarantee that some near optimal kernel $\hat{\pi}$ satisfies condition (M.1), which is also needed in the non-robust case to show the large deviations lower bound.
	\begin{figure}[t]
	\centering
	\vspace{-4cm}
	\hspace{-2.15cm}
	\mbox{\includegraphics[width=.55\textwidth]{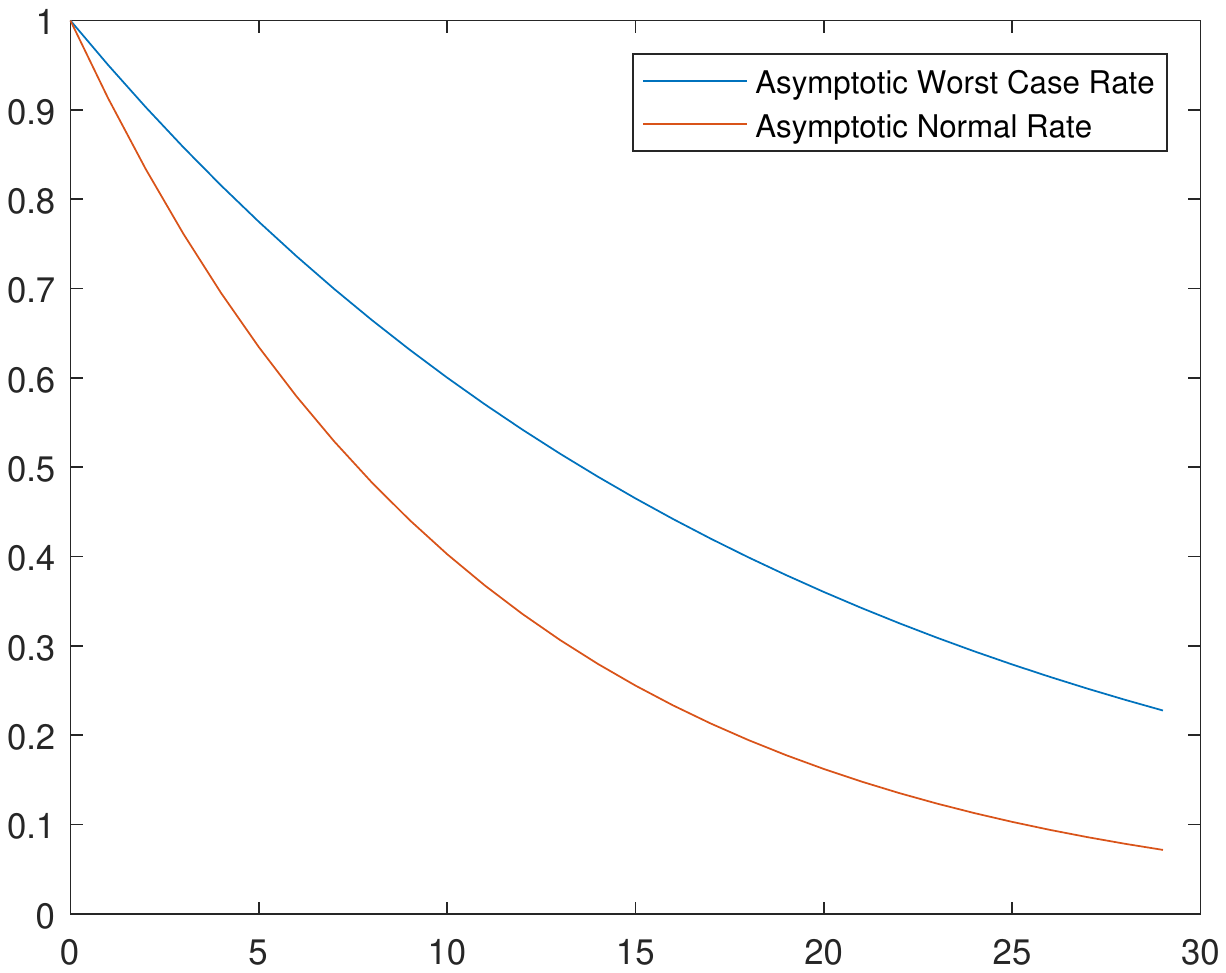}\hspace{-3.17cm}
	\includegraphics[width=.55\textwidth]{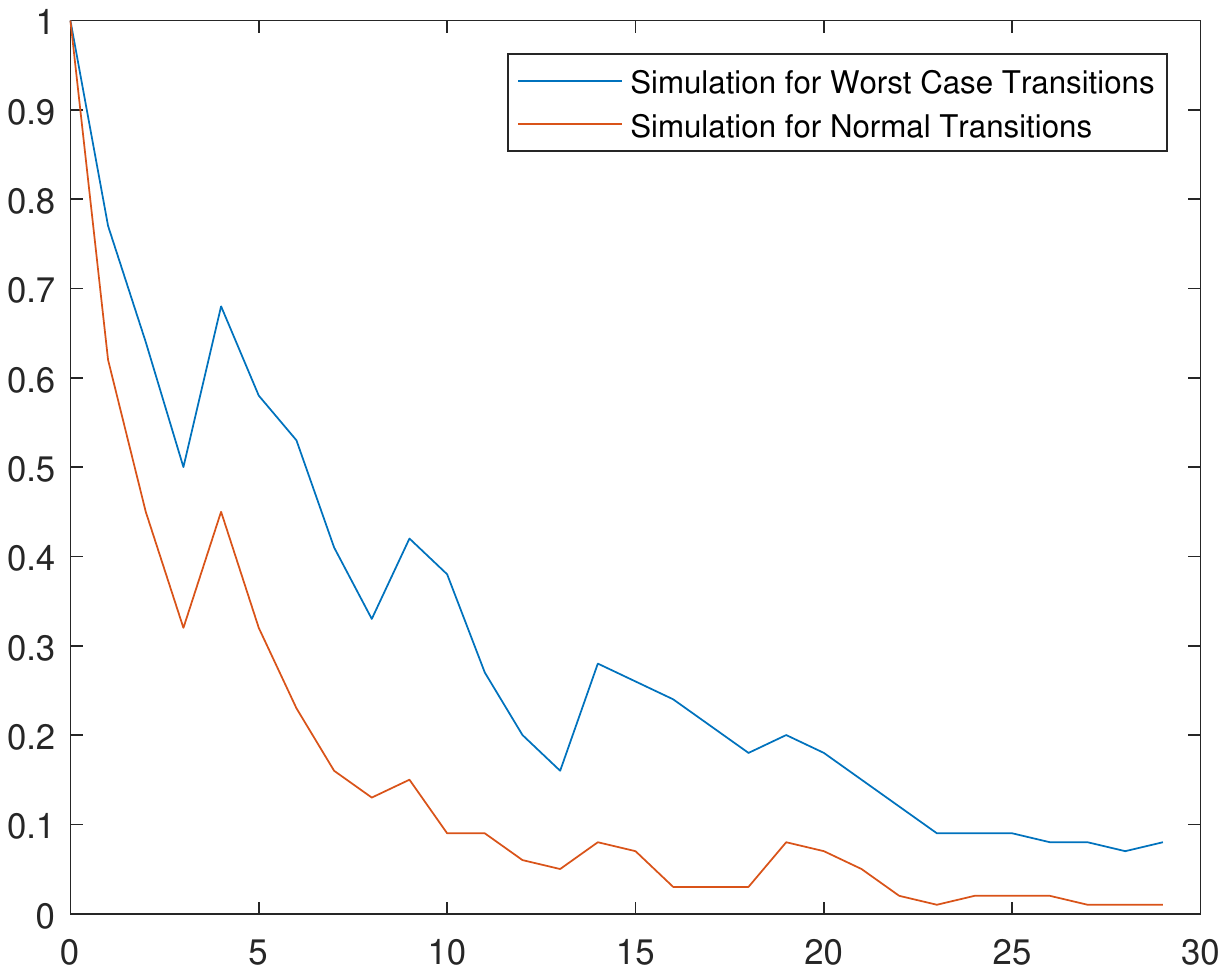}\hspace{-3.17cm}
	\includegraphics[width=.55\textwidth]{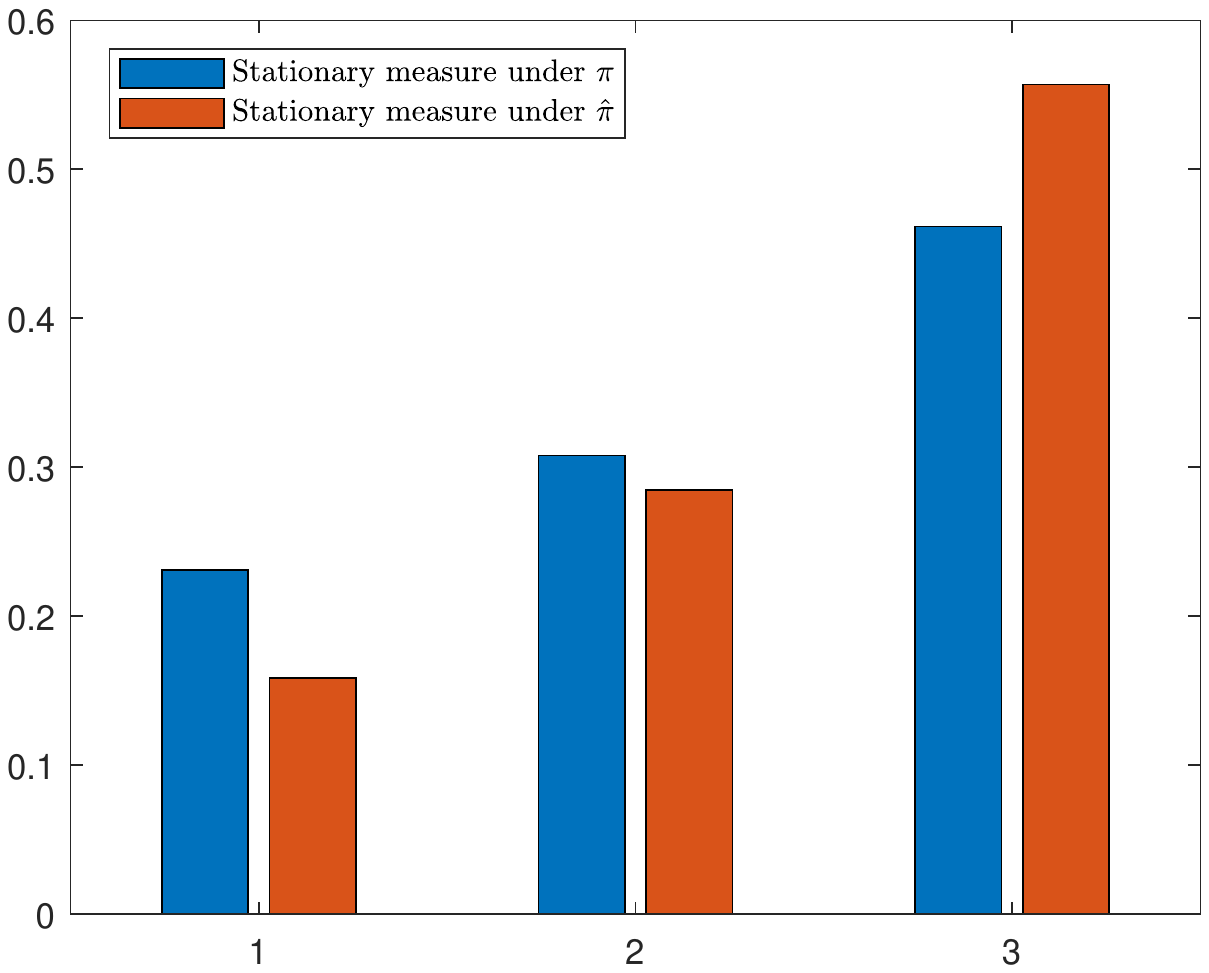}\hspace{-2cm}}
	\label{fig:figure3}
	\vspace{-4cm}
	\caption{Illustration of convergence rates, simulated (100 paths) realized convergence and the stationary distributions under the normal Markov chain and the robust worst-case Markov chain.}
	\end{figure}
\end{example}

\subsection{Robust Weak Law of Large Numbers}
Let $(E,d)$ be compact. 
In this section, Theorem \ref{limitweakrobust} is proven. We first show the upper bound in Theorem \ref{limitweakrobust2} and explain afterwards how to obtain the lower bound.

Up to Theorem \ref{limitweakrobust2}, let $\pi$ satisfy the Feller property. Define
\[
\beta(\mu,\nu) := \left\{ \begin{array}{ll} 0,& \text{if } d_W(\mu,\nu)\leq r, \\ 
\infty,&
\text{else,} \end{array} \right.
\]
for some $r \geq 0$ 
and find
\[
\beta_n^{\theta}(\nu) = \infty \cdot \eins_{\left(M_n(\theta)\right)^C}(\nu).
\]

\begin{lemma}
	\label{weaklimbeta}
	$\beta_2^{\cdot}(\cdot)$ is convex and lower semi-continuous.
	\begin{proof}
		We first show convexity: Let $\theta_1, \theta_2 \in \mathcal{P}(E)$, $\nu_1,\nu_2 \in \mathcal{P}(E^2)$ and $\lambda \in (0,1)$. We have to show
		\[
		\beta_2^{\lambda \theta_1 + (1-\lambda) \theta_2}(\lambda \nu_1 + (1-\lambda) \nu_2) \leq \lambda \beta_2^{\theta_1}(\nu_1) + (1-\lambda) \beta_2^{\theta_2}(\nu_2).
		\]
		To this end, it suffices to show that if the right hand side is zero, the left hand side has to be zero as well. If the right hand side is zero, then
		both $\nu_1 \in M_2(\theta_1)$ and $\nu_2 \in M_2(\theta_2)$. It follows by Lemma \ref{Mnconv} that $\lambda \nu_1 + (1-\lambda)\nu_2 \in M_2(\lambda \theta_1 + (1-\lambda) \theta_2)$ and thus the left hand side is also zero.
		
		We now show lower semi-continuity: Let $(\theta_n,\nu_n) \stackrel{w}{\rightarrow} (\theta,\nu) \in \mathcal{P}(E)\times \mathcal{P}(E^2)$. We have to show
		\[
		\liminf_{n\rightarrow \infty} \beta_2^{\theta_n}(\nu_n) \geq \beta_2^{\theta}(\nu).
		\]
		Without loss of generality, the left hand side is not equal infinity. We have to show that the right hand side is zero. We first choose an arbitrary subsequence and then a further subsequence still denoted by $(\theta_n,\nu_n)_{n\in\mathbb{N}}$ such that for all $n\in \mathbb{N}$
		\begin{align*}
			\beta_2^{\theta_n}(\nu_n) &< \infty,\\
			d_W(\theta_n,\theta) &\leq \frac{1}{n}.
		\end{align*}
		It follows that $\nu_n \in M_2(\theta_n)$ for all $n \in \mathbb{N}$ and with the same notation and argumentation as in the proof of \ref{lscLDP} it follows $\nu \in M_2^{r + \frac{1}{n}}(\theta)$ for all $n\in \mathbb{N}$ and thus $\nu \in M_2(\theta)$, i.e.~$\beta_2^{\theta}(\nu) = 0$.
	\end{proof}
\end{lemma}
By applying Theorem \ref{mainth} and Corollary \ref{cor1} we get:
\begin{theorem}
	\label{limitweakrobust2}
	For all upper semi-continuous and bounded from above functions $F: \mathcal{P}(E) \rightarrow [-\infty,\infty)$ it holds
	\[
	\limsup_{n\rightarrow\infty} \sup_{\mu \in M_n(\pi_0)} \int_{E^n} F\circ L_n d\mu \leq \sup_{\substack{\nu \in \mathcal{P}(E) : \\ \exists q, \nu q = \nu: \nu \otimes q \in M_2(\nu)}} F(\nu).
	\]
\end{theorem}

We now focus on the lower bound in Theorem \ref{limitweakrobust}.
Therefore, let $\pi$ satisfy (M) (but no longer has to satisfy the Feller property). We define 
\[\underline{\beta}(\mu,\nu) := \left\{ \begin{array}{ll} 0,& \text{if } d_W(\mu,\nu)\leq r \text{ and } \mu \ll \nu, \\ 
\infty,&
\text{else,} \end{array} \right.\] so that (B.3) holds. We obtain
\[
\underline{\beta}_n^{\theta}(\nu) = \infty \cdot \eins_{(\underline{M}_n(\theta))^C}(\nu).
\]
Proving (B.1) for $\underline{\beta}$ works completely analogous to the case of $\beta$ in Lemma \ref{weaklimbeta} by replacing $M_n$ by $\underline{M}_n$. Applying Theorem \ref{mainth} yields
\[
\liminf_{n\rightarrow \infty} \sup_{\mu \in \underline{M}_n(\pi_0)} \int_{E^n} F \circ L_n d\mu \geq \sup_{\substack{\nu \in \mathcal{P}(E) : \\ \exists q, \nu q = \nu: \nu \otimes q \in \underline{M}_2(\nu)}} F(\nu)
\]
for all $F\in C_b(\mathcal{P}(E))$.
Theorem \ref{limitweakrobust} is shown.
\bibliographystyle{abbrv}

\begin{thebibliography}{10}
	
	\bibitem{acciaio2011dynamic}
	B.~Acciaio and I.~Penner.
	\newblock Dynamic risk measures.
	\newblock In {\em Advanced mathematical methods for finance}, pages 1--34.
	Springer, 2011.
	
	\bibitem{bartl2016exponential}
	D.~Bartl.
	\newblock Exponential utility maximization under model uncertainty for
	unbounded endowments.
	\newblock {\em arXiv preprint arXiv:1610.00999}, 2016.
	
	\bibitem{bartl2016pointwise}
	D.~Bartl.
	\newblock Pointwise dual representation of dynamic convex expectations.
	\newblock {\em arXiv preprint arXiv:1612.09103}, 2016.
	
	\bibitem{bertsekas2004stochastic}
	D.~P. Bertsekas and S.~Shreve.
	\newblock {\em Stochastic optimal control: the discrete-time case}.
	\newblock Athena Scientific, 1996.
	
	\bibitem{blanchet2016quantifying}
	J.~Blanchet and K.~Murthy.
	\newblock Quantifying distributional model risk via optimal transport.
	\newblock 2016.
	
	\bibitem{breiman1992probability}
	L.~Breiman.
	\newblock Probability, volume 7 of classics in applied mathematics.
	\newblock {\em Society for Industrial and Applied Mathematics (SIAM),
		Philadelphia, PA}, 1992.
	
	\bibitem{cerreia2016ergodic}
	S.~Cerreia-Vioglio, F.~Maccheroni, and M.~Marinacci.
	\newblock Ergodic theorems for lower probabilities.
	\newblock {\em Proceedings of the American Mathematical Society},
	144(8):3381--3396, 2016.
	
	\bibitem{cheridito2011composition}
	P.~Cheridito and M.~Kupper.
	\newblock Composition of time-consistent dynamic monetary risk measures in
	discrete time.
	\newblock {\em International Journal of Theoretical and Applied Finance},
	14(01):137--162, 2011.
	
	\bibitem{de1990large}
	A.~De~Acosta.
	\newblock Large deviations for empirical measures of markov chains.
	\newblock {\em Journal of Theoretical Probability}, 3(3):395--431, 1990.
	
	\bibitem{de2009imprecise}
	G.~De~Cooman, F.~Hermans, and E.~Quaeghebeur.
	\newblock Imprecise markov chains and their limit behavior.
	\newblock {\em Probability in the Engineering and Informational Sciences},
	23(4):597--635, 2009.
	
	\bibitem{dellacherie1982probability}
	C.~Dellacherie and P.-A. Meyer.
	\newblock {\em Probability and potential B: Theory of martingales}.
	\newblock North-Holland, Amsterdam, 1982.
	
	\bibitem{dembo2010large}
	A.~Dembo and O.~Zeitouni.
	\newblock {\em Large deviations techniques and applications, volume 38 of
		Stochastic Modelling and Applied Probability}.
	\newblock Springer-Verlag, Berlin, 2010.
	
	\bibitem{donsker1975asymptotic}
	M.~Donsker and S.~Varadhan.
	\newblock Asymptotic evaluation of certain markov process expectations for
	large time, i.
	\newblock {\em Communications on Pure and Applied Mathematics}, 28(1):1--47,
	1975.
	
	\bibitem{donsker1976asymptotic}
	M.~Donsker and S.~Varadhan.
	\newblock Asymptotic evaluation of certain markov process expectations for
	large time—iii.
	\newblock {\em Communications on pure and applied Mathematics}, 29(4):389--461,
	1976.
	
	\bibitem{dudley2002real}
	R.~M. Dudley.
	\newblock {\em Real analysis and probability}, volume~74.
	\newblock Cambridge University Press, 2002.
	
	\bibitem{dupuis2011weak}
	P.~Dupuis and R.~S. Ellis.
	\newblock {\em A weak convergence approach to the theory of large deviations},
	volume 902.
	\newblock John Wiley \& Sons, 2011.
	
	\bibitem{esfahani2015data}
	P.~M. Esfahani and D.~Kuhn.
	\newblock Data-driven distributionally robust optimization using the
	wasserstein metric: Performance guarantees and tractable reformulations.
	\newblock {\em arXiv preprint arXiv:1505.05116}, 2015.
	
	\bibitem{ethier1986markov}
	S.~N. Ethier and T.~G. Kurtz.
	\newblock {\em Markov Processes}.
	\newblock John Wiley \& Sons Inc, 1986.
	
	\bibitem{gao2016distributionally}
	R.~Gao and A.~J. Kleywegt.
	\newblock Distributionally robust stochastic optimization with wasserstein
	distance.
	\newblock {\em arXiv preprint arXiv:1604.02199}, 2016.
	
	\bibitem{gibbs2002choosing}
	A.~L. Gibbs and F.~E. Su.
	\newblock On choosing and bounding probability metrics.
	\newblock {\em International statistical review}, 70(3):419--435, 2002.
	
	\bibitem{hanasusanto2015distributionally}
	G.~A. Hanasusanto, V.~Roitch, D.~Kuhn, and W.~Wiesemann.
	\newblock A distributionally robust perspective on uncertainty quantification
	and chance constrained programming.
	\newblock {\em Mathematical Programming}, 151(1):35--62, 2015.
	
	\bibitem{hartfiel1994theory}
	D.~Hartfiel and E.~Seneta.
	\newblock On the theory of markov set-chains.
	\newblock {\em Advances in Applied Probability}, 26(4):947--964, 1994.
	
	\bibitem{hartfiel2006markov}
	D.~J. Hartfiel.
	\newblock {\em Markov set-chains}.
	\newblock Springer, 2006.
	
	\bibitem{jain1990large}
	N.~C. Jain.
	\newblock Large deviation lower bounds for additive functionals of markov
	processes.
	\newblock {\em The Annals of Probability}, pages 1071--1098, 1990.
	
	\bibitem{kurano1998controlled}
	M.~Kurano, J.~Song, M.~Hosaka, and Y.~Huang.
	\newblock Controlled markov set-chains with discounting.
	\newblock {\em Journal of applied probability}, 35(2):293--302, 1998.
	
	\bibitem{lacker2015law}
	D.~Lacker.
	\newblock Law invariant risk measures and information divergences.
	\newblock {\em arXiv preprint arXiv:1510.07030}, 2015.
	
	\bibitem{lacker2016non}
	D.~Lacker.
	\newblock A non-exponential extension of sanov's theorem via convex duality.
	\newblock {\em arXiv preprint arXiv:1609.04744}, 2016.
	
	\bibitem{lan2017strong}
	Y.~Lan and N.~Zhang.
	\newblock Strong limit theorems for weighted sums of negatively associated
	random variables in nonlinear probability.
	\newblock {\em arXiv preprint arXiv:1706.05788}, 2017.
	
	\bibitem{ney1987markov}
	P.~Ney, E.~Nummelin, et~al.
	\newblock Markov additive processes ii. large deviations.
	\newblock {\em The Annals of Probability}, 15(2):593--609, 1987.
	
	\bibitem{peng2009survey}
	S.~Peng.
	\newblock Survey on normal distributions, central limit theorem, brownian
	motion and the related stochastic calculus under sublinear expectations.
	\newblock {\em Science in China Series A: Mathematics}, 52(7):1391--1411, 2009.
	
	\bibitem{peng2010nonlinear}
	S.~Peng.
	\newblock Nonlinear expectations and stochastic calculus under uncertainty.
	\newblock {\em arXiv preprint arXiv:1002.4546}, 2010.
	
	\bibitem{vskulj2006finite}
	D.~{\v{S}}kulj.
	\newblock Finite discrete time markov chains with interval probabilities.
	\newblock {\em Soft Methods for Integrated Uncertainty Modelling}, 37:299--306,
	2006.
	
	\bibitem{vskulj2009discrete}
	D.~{\v{S}}kulj.
	\newblock Discrete time markov chains with interval probabilities.
	\newblock {\em International journal of approximate reasoning},
	50(8):1314--1329, 2009.
	
	\bibitem{stroock1993probability}
	D.~W. Stroock.
	\newblock {\em Probability theory}.
	\newblock Cambridge Univ. Press, 1993.
	
	\bibitem{villani2008optimal}
	C.~Villani.
	\newblock {\em Optimal transport: old and new}, volume 338.
	\newblock Springer Science \& Business Media, 2008.
	
\end{thebibliography}

\end{document}